\documentclass[10pt,a4paper]{article}

\usepackage{amsmath,amssymb,amsfonts,amsthm}
\usepackage{mathrsfs}
\usepackage{graphicx}
\usepackage{booktabs}
\usepackage{multirow}
\usepackage{xcolor}
\usepackage{enumerate}
\usepackage{appendix}
\usepackage{hyperref}

\newtheorem{theorem}{Theorem}[section]

\newtheorem{lemma}[theorem]{Lemma}

\theoremstyle{definition}
\newtheorem{definition}[theorem]{Definition}
\newtheorem{example}[theorem]{Example}

\theoremstyle{remark}
\newtheorem{remark}[theorem]{Remark}

\title{\bfseries
B.-Y. Chen inequalities for Riemannian submersion and their
applications
}

\author{Ravindra Singh, Mukut Mani Tripathi }

\date{}

\begin{document}
\maketitle

%------------------------------------------------
\begin{abstract}
\noindent In this paper, we introduce B.-Y. Chen inequalities for Riemannian
submersions between Riemannian manifolds. We derive these inequalities for
vertical, horizontal, and mixed distributions, establishing relationships
between intrinsic invariants and extrinsic invariants. We also investigate
the corresponding equality cases. As applications, the results are obtained
for submersions whose total space is a real, complex, generalized Sasakian
space form. Several examples are provided to illustrate both equality and
strict inequality cases.
\end{abstract}



\section{Introduction}\label{sec1}
According to Nash's immersion theorem \cite{Nash_56_AM}, every $n$%
-dimensional Riemannian manifold admits an isometric immersion into the
Euclidean space ${\Bbb E}^{n\left( n+1\right) \left( 3n+11\right) /2}$.
Thus, any Riemannian manifold can be regarded as a submanifold of a
Euclidean space, which provides a natural motivation for the study of
submanifolds of Riemannian manifolds. One of the fundamental problems in
submanifold theory is to establish relationships between the main intrinsic
invariants (scalar and sectional curvatures) and the main extrinsic
invariants (squared mean curvature) of a submanifold. In this direction,
classical results such as the Gauss--Bonnet theorem, the isoperimetric
inequality, and the Chern--Lashof theorem provide important relation between
intrinsic and extrinsic invariants for submanifolds in Euclidean spaces.

In \cite{Chen_93_AM}, B.-Y. Chen established a sharp inequality for
submanifolds of real space forms involving main intrinsic invariants and the
main extrinsic invariant. He also obtained the corresponding inequality for
submanifolds of complex space forms \cite{Chen_96_AM}. Hong {\it et al.} 
\cite{Hong_Matsumoto_Tripathi_05_SUTJM} extended this inequality to
submanifolds of Riemannian manifolds. This inequality has been extensively
studied in the survey article \cite{Chen_23_Math} and the references cited
therein. \newline
O'Neill \cite{Neill_66_MMJ} introduced Riemannian submersions, which provide
a natural framework for relating the geometry of two Riemannian manifolds
and have important applications in differential geometry. Motivated by
Chen's inequalities for submanifolds, several authors have established
analogous results for Riemannian submersions by relating intrinsic
invariants with extrinsic invariants induced by the submersion. In
particular, inequalities involving vertical, horizontal, and mixed
(horizontal and vertical) distributions, together with their equality cases,
have attracted considerable attention. In this paper, we derive B.-Y. Chen's
inequalities for Riemannian submersions between Riemannian manifolds. These
inequalities are established for vertical, horizontal, and mixed
distributions, and the corresponding equality cases are thoroughly
investigated. As applications, we obtain results for Riemannian submersions
whose total space is a real space form, a complex space form, and a
generalized Sasakian space form. Furthermore, explicit examples are
constructed to demonstrate both the equality and strict inequality cases. 
\newline
This paper is organized as follows. In Section \ref{section 2}, we recall the basic definitions, preliminaries, and fundamental results concerning Riemannian submersions and curvature tensors that will be used throughout the paper. In Section \ref{section 3}, we establish a B.-Y. Chen inequality for Riemannian submersions associated with the vertical distribution and characterize the corresponding equality cases. Furthermore, we apply the obtained result to Riemannian submersions whose total space is a real space form, a complex space form, and a generalized Sasakian space form, together with suitable illustrative examples. Section \ref{section_Hor} is devoted to the derivation of the corresponding inequality related to the horizontal distribution, along with analogous applications and examples in the above geometric settings. Finally, in Section \ref{Section 4}, we derive the corresponding inequality involving mixed distributions and discuss its applications supported by illustrative examples.

\section{Preliminaries\label{section 2}}

Let $\left( M_{1},g_{1}\right) $ and $\left( M_{2},g_{2}\right) $ be two
Riemannian manifolds of dimension $n$ and $m$, respectively. A surjective
smooth map $\pi :\left( M_{1},g_{1}\right) \rightarrow \left(
M_{2},g_{2}\right) $ is called a {\em Riemannian submersion} if its
differential map $\pi _{\ast p}:T_{p}M_{1}\rightarrow T_{\pi (p)}M_{2}$ is
surjective for all $p\in M_{1}$ and $\pi _{\ast p}$ preserves the length of
all horizontal vectors at $p$ \cite{Neill_66_MMJ}. The tangent vectors to
the fibers are vertical, while the orthogonal vectors to the fibers are
horizontal. Thus $TM_{1}$, the decomposition is a direct sum of two
distributions: the vertical distribution ${\cal V}=\ker \pi _{\ast }$ and
its orthogonal complement (known as the horizontal distribution) ${\cal H}%
=\left( \ker \pi _{\ast }\right) ^{\perp }$. Similarly, for each $p\in M_{1}$
the vertical and horizontal spaces in $T_{p}M_{1}$ are denoted by ${\cal V}%
_{p}=\left( \ker \pi _{\ast }\right) _{p}$ and ${\cal H}_{p}=\left( \ker \pi
_{\ast }\right) _{p}^{\perp }$, respectively. \newline
Let $\{V_{1},\ldots ,V_{r},U_{1},\ldots ,U_{s}\}$ be an orthonormal basis of 
$T_{p}M_{1}$ such that $\{V_{1},\ldots ,V_{r}\}$ and $\{U_{1},\ldots
,U_{s}\} $ are orthonormal bases of the vertical space ${\cal V}_{p}$ and
the horizontal space ${\cal H}_{p}$, respectively. Consider the planes $\Pi =%
{\rm span}\{V_{1},V_{2}\}$ and ${\Bbb P}={\rm span}\{U_{1},U_{2}\}$. We
shall use these bases throughout this paper.

\subsection{O'Neill tensors}

The geometry of Riemannian submersions is characterized by O'Neill's tensors 
{${\cal T}$} and {${\cal A}$} defined for all vector fields $E$, $F$ on $M_{1}$
by 
\begin{eqnarray*}
{\cal T}\left( E,F\right) &=&{{\cal T}_{E}F=h\nabla _{vE}vF+v\nabla _{vE}hF},
\\
{\cal A}\left( E,F\right) &=&{\cal A}_{E}F=v{\nabla _{hE}hF+h\nabla _{hE}vF,}
\end{eqnarray*}%
where $\nabla $ is the Levi-Civita connection on $g_{1}$, $h$ and $v$ are
projection morphisms of $E$, $F\in TM_{1}$ to ${\cal H}$ and ${\cal V}$,
respectively. We also have%
\begin{eqnarray*}
{{\cal T}^{{\cal H}}}:{\cal V}\times {{\cal V}\rightarrow {\cal H}}, \\
{{\cal T}^{{\cal V}}}{:{\cal V}\times {{\cal H}\rightarrow {\cal V}}}, \\
{{\cal A}^{{\cal H}}}{:{\cal H}\times {{\cal V}\rightarrow {\cal H}}}, \\
{{\cal A}^{{\cal V}}}{:{\cal H}\times {{\cal H}\rightarrow {\cal V}}},
\end{eqnarray*}%
and 
\[
g_{1}\left( {{\cal T}^{{\cal H}}}\left( V_{i},V_{j}\right) ,U_{\alpha
}\right) =\left( {{\cal T}^{{\cal H}}}\right) _{ij}^{\alpha },\quad
g_{1}\left( {{\cal A}^{{\cal V}}}\left( U_{\alpha },U_{\beta }\right)
,V_{i}\right) =\left( {{\cal A}^{{\cal V}}}\right) _{\alpha \beta }^{i}, 
\]%
where $V_{i},V_{j}\in {\cal V}$, $U_{\alpha },U_{\beta }\in {\cal H}$. The O'Neill tensors also satisfy: 
\[
{\cal A}_{U_{i}}^{{\cal V}}U_{j}=-{\cal A}_{U_{j}}^{{\cal V}}U_{i},\quad 
{\cal T}_{V_{i}}^{{\cal H}}V_{j}={\cal T}_{V_{j}}^{{\cal H}}V_{i}, 
\]%
and 
\[
g_{1}\left( {\cal T}_{E}F,G\right) =-g_{1}\left( F,{\cal T}_{E}G\right)
,\quad g_{1}\left( {\cal A}_{E}F,G\right) =-g_{1}\left( F,{\cal A}%
_{E}G\right) , 
\]%
where $V_{i},V_{j}\in {\cal V}$, $U_{i},U_{j}\in {\cal H}$ and $E$,
\thinspace $F$, $G\in TM_{1}$ \cite{Neill_66_MMJ}. The {\it mean curvature
vector field }$H${\it \ of the fibers} of $\pi $ is defined as \cite%
{Falcitelli_04_WS} 
\begin{equation}
H(p)=\frac{1}{r}N,\quad N=\sum_{i=1}^{r}{\cal T}^{{\cal H}}\left(
V_{i},V_{i}\right) .  \label{eq-P-(9)}
\end{equation}%
The horizontal divergence of any vector field $U$ on ${\cal H}$ is given by $%
\breve{\delta}\left( U\right) $ and defined by%
\[
\breve{\delta}\left( U\right) =\sum_{j=1}^{r}\sum_{i=1}^{s}g\left( \nabla
_{U_{i}}U,U_{i}\right) . 
\]%
From \cite{Gulbahar_Meric_Kilic_17_KJM}, we have 
\begin{equation}
\breve{\delta}\left( N\right) =\sum_{j=1}^{r}\sum_{i=1}^{s}g\left( \left(
\nabla _{U_{i}}{\cal T}\right) \left( V_{j},V_{j}\right) ,U_{i}\right) .
\label{eq-P-(9.3.1)}
\end{equation}

\subsection{Relations between Riemann curvature tensor fields}

Let $R^{M_{1}}$, $R^{M_{2}}$, $R^{{\cal V}}$, and $R^{{\cal H}}$ denote the
Riemann curvature tensor fields corresponding to $M_{1}$, $M_{2}$, ${\cal V}$%
, and ${\cal H}$, respectively. Then, we have 
\begin{eqnarray}
R^{M_{1}}\left( F_{1},F_{2},F_{3},F_{4}\right) &=&R^{{\cal V}}\left(
F_{1},F_{2},F_{3},F_{4}\right) +g_{1}\left( {\cal T}_{F_{1}}F_{4},{\cal T}%
_{F_{2}}F_{3}\right)  \nonumber \\
&&-g_{1}\left( {\cal T}_{F_{2}}F_{4},{\cal T}_{F_{1}}F_{3}\right) ,
\label{eq-P-(10)}
\end{eqnarray}%
\begin{eqnarray}
R^{M_{1}}\left( X,Y,Z,W\right) &=&R^{{\cal H}}\left( X,Y,Z,W\right)
-2g_{1}\left( {\cal A}_{X}Y,{\cal A}_{Z}W\right)  \nonumber \\
&&+g_{1}\left( {\cal A}_{Y}Z,{\cal A}_{X}W\right) -g_{1}\left( {\cal A}_{X}Z,%
{\cal A}_{Y}W\right) ,  \label{eq-P-(11)}
\end{eqnarray}%
\begin{eqnarray}
R^{M_{1}}\left( X,F_{1},Y,F_{2}\right) &=&g_{1}\left( \left( \nabla _{X}%
{\cal T}\right) \left( F_{1},F_{2}\right) ,Y\right) +g_{1}\left( \left(
\nabla _{F_{1}}{\cal A}\right) \left( X,Y\right) ,F_{2}\right)  \nonumber \\
&&-g_{1}\left( {\cal T}_{F_{1}}X,{\cal T}_{F_{2}}Y\right) +g_{1}\left( {\cal %
A}_{Y}F_{2},{\cal A}_{X}F_{1}\right)  \label{eq-P-(12)}
\end{eqnarray}%
for all $X,Y,Z,W\in {\cal H}$\ and $F_{1},F_{2},F_{3},F_{4}\in {\cal V}$.
Here, $\nabla $ is the Levi-Civita connection with respect to the metric $%
g_{1}$ \cite{Falcitelli_04_WS,Neill_66_MMJ}.
 
In the sequel, we use the following notations throughout the paper. 
\begin{equation}
2\tau _{{\cal V}}^{{\cal V}}=\sum_{i,j=1}^{r}R^{{\cal V}}\left(
V_{i},V_{j},V_{j},V_{i}\right) ,\ 2\tau _{{\cal H}}^{{\cal H}%
}=\sum_{i,j=1}^{s}R^{{\cal V}}\left( U_{i},U_{j},U_{j},U_{i}\right), 
\label{eq-P-(9.1)}
\end{equation}
\begin{equation}
K_{{\cal V}}^{{\cal V}}\left( \Pi \right) =R^{{\cal V}}\left(
V_{1},V_{2},V_{2},V_{1}\right) ,\ K_{{\cal V}}^{M_{1}}\left( \Pi \right)
=R^{M_{1}}\left( V_{1},V_{2},V_{2},V_{1}\right) ,  \label{eq-P-(9.2)}
\end{equation}%
\begin{equation}
K_{{\cal H}}^{{\cal H}}\left( \Pi \right) =R^{{\cal H}}\left(
U_{1},U_{2},U_{2},U_{1}\right) ,\ K_{{\cal H}}^{M_{1}}\left( \Pi \right)
=R^{M_{1}}\left( U_{1},U_{2},U_{2},U_{1}\right) ,  \label{eq-P-(9.2.0)}
\end{equation}%
\begin{equation}
2\tau _{{\cal V}}^{M_{1}}(p)=\sum\limits_{i,j=1}^{r}R^{M_{1}}\left(
V_{i},V_{j},V_{j},V_{i}\right) ,\ 2\tau _{{\cal H}}^{M_{1}}(p)=\sum%
\limits_{i,j=1}^{s}R^{M_{1}}\left( U_{i},U_{j},U_{j},U_{i}\right) ,
\label{eq-P-(9.2.1)}
\end{equation}%
\begin{equation}
\tau ^{M_{1}}(p)=\tau _{{\cal V}}^{M_{1}}(p)+\tau _{{\cal H}%
}^{M_{1}}(p)+\sum_{i=1}^{s}\sum_{j=1}^{r}R^{M_{1}}\left(
U_{i},V_{j},V_{j},U_{i}\right) ,  \label{eq-scal-M1}
\end{equation}%
\begin{equation}
\left\Vert {\cal T}^{{\cal V}}\right\Vert
^{2}=\sum_{j=1}^{r}\sum_{i=1}^{s}g\left( {\cal T}_{V_{j}}^{{\cal V}}U_{i},%
{\cal T}_{V_{j}}^{{\cal V}}U_{i}\right) ,\ \left\Vert {\cal T}^{{\cal H}%
}\right\Vert ^{2}=\sum_{i,j=1}^{r}g\left( {\cal T}_{V_{i}}^{{\cal V}}V_{j},%
{\cal T}_{V_{i}}^{{\cal V}}V_{j}\right),   \label{eq-P-(14)}
\end{equation}%
\begin{equation}
\left\Vert {\cal A}^{{\cal V}}\right\Vert ^{2}=\sum_{i,j=1}^{s}g\left(
A_{U_{i}}^{{\cal V}}U_{j},A_{U_{i}}^{{\cal V}}U_{j}\right) ,\quad \left\Vert 
{\cal A}^{{\cal H}}\right\Vert ^{2}=\sum_{j=1}^{r}\sum_{i=1}^{s}g\left(
A_{U_{i}}^{{\cal H}}V_{j},A_{U_{i}}^{{\cal H}}V_{j}\right) ,
\label{eq-P-(14.1)}
\end{equation}%
\[
\left\Vert Q\right\Vert ^{2}=\sum\limits_{i=1}^{r}\Vert QV_{i}\Vert
^{2}=\sum\limits_{i,j=1}^{r}\left( g_{2}(QV_{i},V_{j})\right) ^{2},
\]%
\[
\left\Vert P\right\Vert ^{2}=\sum\limits_{i=1}^{s}\Vert PU_{i}\Vert
^{2}=\sum\limits_{i,j=1}^{s}\left( g_{2}(PU_{i},U_{j})\right) ^{2},
\]%

\[
\left\Vert P^{{\cal V}}\right\Vert ^{2}=\sum\limits_{i=1}^{r}\Vert
PV_{i}\Vert ^{2}=\sum\limits_{i=1}^{r}\sum\limits_{j=1}^{s}\left(
g_{2}(PV_{i},U_{j})\right) ^{2}.
\]

\begin{definition}
The Kulkarni-Nomizu product $T_{1}\circledast T_{2}$ of $\left( 0,2\right) $%
-tensor fields $T_{1}$ and $T_{2}$ in a smooth manifold $M$ is a $\left(
0,4\right) $-tensor field defined by \cite{Tripathi_25_CMAMS} 
\begin{eqnarray}
\left( T_{1}\circledast T_{2}\right) \left( X,Y,Z,W\right)  &=&T_{1}\left(
Y,Z\right) T_{2}\left( X,W\right) -T_{1}\left( X,Z\right) T_{2}\left(
Y,W\right)   \nonumber \\
&&+T_{2}\left( Y,Z\right) T_{1}\left( X,W\right) -T_{2}\left( X,Z\right)
T_{1}\left( Y,W\right)   \label{eq-KN-product}
\end{eqnarray}%
for all vector fields $X$, $Y$, $Z$, $W$ on $M$.
\end{definition}
\begin{definition}
The symmetric product of any two $\left( 0,2\right) $-tensor fields $T_{1}$
and $T_{2}$ is a $\left( 0,4\right) $-tensor field $T_{1}\circledcirc T_{2}$
defined by \cite{Tripathi_25_CMAMS} 
\begin{equation}
T_{1}\circledcirc T_{2}=T_{1}\circledast T_{2}+T_{2}\circledast T_{1}.
\label{eq-KN-symm-product}
\end{equation}
\end{definition}
Now, let $\left( M,g\right) $ be an $n$-dimensional Riemannian manifold. Let 
$T$ be a Kulkarni-Nomizu tensor field so that it satisfies%
\begin{equation}
T\left( X,Y,Z,W\right) =-T\left( Y,X,Z,W\right) ,  \label{eq-KN-T-1}
\end{equation}%
\begin{equation}
T\left( X,Y,Z,W\right) =-T\left( X,Y,W,Z\right) ,  \label{eq-KN-T-2}
\end{equation}%
\begin{equation}
T\left( X,Y,Z,W\right) =T\left( Z,W,X,Y\right) ,  \label{eq-KN-T-3}
\end{equation}%
\begin{equation}
T\left( X,Y,Z,W\right) +T\left( Y,Z,X,W\right) +T\left( Z,X,Y,W\right) =0,
\label{eq-KN-T-4}
\end{equation}%
\begin{equation}
T\left( X,Y,Z,W\right) +T\left( X,Z,W,Y\right) +T\left( X,W,Y,Z\right) =0
\label{eq-KN-T-5}
\end{equation}%
for all vector fields $X$, $Y$, $Z$ and $W$ on $M$. It is observed that if $%
T $ satisfies any two of the three conditions (\ref{eq-KN-T-1}), (\ref%
{eq-KN-T-2}), (\ref{eq-KN-T-3}) and any one of the two conditions (\ref%
{eq-KN-T-4}), (\ref{eq-KN-T-5}), then it also satisfies the remainning two
relations.
\begin{definition}
{\rm (\cite[Kobayashi and Nomizu 1963, p. 209]{Kobayashi_Nomizu_1963}, \cite[%
Takahashi 1972]{Takahashi_72_KJSM})} A Riemannian manifold $\left(
M,g\right) $ with constant sectional curvature $c$ is called a real space
form, and its Riemann-Christoffel curvature tensor field $R^{M}$ is given by 
{\rm \cite[Tripathi, 2026, p. 29]{Tripathi_2026}} % \begin{equation}
% R^{M}(X,Y)Z=c\{g(Y,Z)
% X-g(X,Z)Y\}  \label{eq-RSF}
% \end{equation}
% for all vector fields $X,Y,Z\in \Gamma (TM)$.
% Using Kulkarni-Nomizu product given by () and the symmetric product given by (), the expression () can be written as
\begin{equation}
R^{M}=\frac{c}{2}\left( g\circledast g\right).  \label{eq-RSF}
\end{equation}
\end{definition}
\begin{definition}
{\rm \cite[Ogiue 1972]{Ogiue_72_JMSJ}} Let $M$ be an almost Hermitian
manifold with an almost Hermitian structure $\left( J,g\right) $. Then $M$
becomes a K\"{a}hler manifold if $\nabla J=0$. A k\"{a}hler manifold with
constant holomorphic sectional curvature $c$ is called a complex space form $%
M\left( c\right) $, and its Riemann-Christoffel curvature tensor field is
given by {\rm \cite[Tripathi, 2026 p. 79]{Tripathi_2026}} % \begin{align}
\begin{equation}
R^{M}=\frac{c}{8}\left( g\circledast g\right) +\frac{c}{4}\left\{ \frac{1}{2}%
\left( J^{b}\circledast J^{b}\right) -\left( J^{b}\circledcirc J^{b}\right)
\right\},  \label{eq-GCSF}
\end{equation}
where $J^{\flat}$ denotes the $(0,2)$-tensor field associated with the $%
(1,1) $-tensor field $J$, defined by 
\[
J^{\flat}(X,Y)=g(X,JY) 
\]
for all vector fields $X$, $Y$ on $M$. Moreover, for any vector field $X$ on 
$M$, we write 
\begin{equation}
JX=PX+QX,  \label{decompose_GCSF_RM}
\end{equation}%
where $PX\in {\cal H}$, $QX\in {\cal V}$.
\end{definition}
\begin{definition}
{\rm (\cite[Alegre et al. 2004]{Alegre_Blair_Carriazo_04_IJM}, \cite[Blair
2010]{Blair_10})} An almost contact metric manifold $M$ equipped with an
almost contact metric structure $\left( \varphi ,\xi ,\eta ,g\right) $ is
called a generalized Sasakian space form, denoted by $M\left(
c_{1},c_{2},c_{3}\right) $, if its Riemann-Christoffel curvature tensor
field $R^{M}$ satisfies {\rm \cite[Tripathi, 2026, p. 115]{Tripathi_2026}} 
\begin{equation}
R^{M}=\frac{1}{2}c_{1}\left( g\circledast g\right) +c_{2}\left\{ \frac{1}{2}%
\left( \varphi ^{b}\circledast \varphi ^{b}\right) -\left( \varphi
^{b}\circledcirc \varphi ^{b}\right) \right\} -c_{3}\left( g\circledast
\left( \eta \otimes \eta \right) \right) ,  \label{eq-GSSF}
\end{equation}%
where $\varphi ^{\flat }$ denotes the $(0,2)$-tensor field associated with
the $(1,1)$-tensor field $\varphi $, defined by 
\[
\varphi ^{\flat }(X,Y)=g(X,\varphi Y)
\]%
for all vector fields $X,Y$ on $M$ \noindent Moreover, for any vector field $%
X$on $M$, we write 
\begin{equation}
\varphi X=PX+QX,  \label{decompose_GSSF_RM}
\end{equation}%
where $PX\in {\cal H}$, $QX\in {\cal V}$.
\end{definition}
\subsection{$\protect\delta \left( 2\right) $-invariant}
B.-Y. Chen introduced a new Riemannian invariant $\delta$ (now known as $%
\delta \left( 2\right) $), defined as the difference between the scalar
curvature and the infimum of sectional curvatures \cite%
{Chen_93_AM,Chen_95_RM} 
\begin{equation}
\delta \left( 2\right) =\tau \left( p\right) -\left( \inf K\right) \left(
p\right),  \label{eq-delta-01}
\end{equation}
where 
\[
\left( \inf K\right) \left( p\right) =\inf \left\{ K\left( \Pi_{2} \right) \
|\ \Pi_{2} \subset T_{p}M,\ \dim \Pi_{2} =2\right\} . 
\]
Similarly, we define {\rm (cf.~\cite[Chen 2000]{Chen_00_JJM})} 
\begin{equation}
\hat{\delta}\left( 2\right) =\tau \left( p\right) -\left( \sup K\right)
\left( p\right),  \label{eq-delta-02}
\end{equation}
where 
\[
\left( \sup K\right) \left( p\right) =\sup \left\{ K\left( \Pi_{2} \right) \
|\ \Pi_{2} \subset T_{p}M,\ \dim \Pi_{2} =2\right\} . 
\]
\begin{lemma}
Let $\pi :(M_{1},g_{1})\rightarrow (M_{2},g_{2})$ be a Riemannian submersion
between Riemannian manifolds with $\dim M_{1}=n$ and $\dim M_{2}=m$. If $%
\dim {\cal V}_{p}=r>2$, then for $2$-planes $\Pi \subset {\cal V}_{p}$ and $%
{\Bbb P}\subset {\cal H}_{p}$, 
\begin{equation}
\delta ^{{\cal V}}\left( 2\right) =\tau _{{\cal V}}^{{\cal V}}\left(
p\right) -\left( \inf K_{{\cal V}}^{{\cal V}}\right) \left( p\right) ,
\label{eq-delta-v-01}
\end{equation}%
where 
\[
\left( \inf K_{{\cal V}}^{{\cal V}}\right) \left( p\right) =\inf \left\{ K_{%
{\cal V}}^{{\cal V}}\left( \Pi \right) \ |\ \Pi \subset {\cal V}_{p},\ \dim
\Pi =2\right\} , 
\]%
and 
\begin{equation}
\delta ^{{\cal H}}\left( 2\right) =\tau _{{\cal H}}^{({\ker \pi _{\ast p}}%
)^{\perp }}\left( p\right) -\left( \inf K_{{\cal H}}^{({\ker \pi _{\ast p}}%
)^{\perp }}\right) \left( p\right) ,  \label{eq-delta-H-02}
\end{equation}%
where 
\[
\left( \inf K_{{\cal H}}^{({\ker \pi _{\ast p}})^{\perp }}\right) \left(
p\right) =\inf \left\{ K_{{\cal H}}^{({\ker \pi _{\ast p}})^{\perp }}\left( 
{\Bbb P}\right) \ |\ {\Bbb P}\subset {\cal H}_{p},\ \dim {\Bbb P}=2\right\}
. 
\]%
Similarly, we define {\rm (cf.~\cite[Chen 2000]{Chen_00_JJM})} 
\begin{equation}
\hat{\delta}^{{\cal V}}\left( 2\right) =\tau _{{\cal V}}^{\ker \pi _{\ast
p}}\left( p\right) -\left( \sup K_{{\cal V}}^{\ker \pi _{\ast p}}\right)
\left( p\right) ,  \label{eq-delta-V-02}
\end{equation}%
where 
\[
\left( \sup K_{{\cal V}}^{\ker \pi _{\ast p}}\right) \left( p\right) =\sup
\left\{ K_{{\cal V}}^{\ker \pi _{\ast p}}\left( \Pi \right) \ |\ \Pi \subset 
{\cal V}_{p},\ \dim \Pi =2\right\} . 
\]%
and 
\begin{equation}
\hat{\delta}^{{\cal H}}\left( 2\right) =\tau _{{\cal H}}^{({\ker \pi _{\ast
p}})^{\perp }}\left( p\right) -\left( \sup K_{{\cal H}}^{({\ker \pi _{\ast p}%
})^{\perp }}\right) \left( p\right) ,  \label{eq-delta-V-03}
\end{equation}%
where 
\[
\left( \sup K_{{\cal H}}^{({\ker \pi _{\ast p}})^{\perp }}\right) \left(
p\right) =\sup \left\{ K_{{\cal H}}^{({\ker \pi _{\ast p}})^{\perp }}\left( 
{\Bbb P}\right) \ |\ {\Bbb P}\subset {\cal H}_{p},\ \dim {\Bbb P}=2\right\}
. 
\]
\end{lemma}
\begin{lemma}
{\rm \cite[Lemma 3.1, Chen 1993]{Chen_93_AM}} \label{Lemma 2}If $k>2$ and $%
a_{1},\ldots ,a_{k},b$ are real numbers such that 
\begin{equation}
\left( \sum_{i=1}^{k}a_{i}\right) ^{2}=\left( k-1\right) \left(
\sum_{i=1}^{k}a_{i}^{2}+b\right)   \label{eq-P-(13)}
\end{equation}
then 
\[
2a_{1}a_{2}\geq b. 
\]
The equality holds if and only if $a_{1}+a_{2}=a_{3}=\cdots =a_{k}$.\label{sec 3}
\end{lemma}
\section{B.-Y. Chen inequality for Riemannian submersion along vertical
distributions \label{section 3}}
We begin with the following:
\begin{theorem}
\label{Theorem 1}Let $\pi :(M_{1},g_{1})\rightarrow (M_{2},g_{2})$ be a
Riemannian submersion between two Riemannian manifolds with $\dim M_{1}=n$
and $\dim M_{2}=m$. If $\dim {\cal V}_{p}=r>2$, then for any $2$-plane $\Pi
\subset {\cal V}_{p}$, 
\begin{equation}
\tau _{{\cal V}}^{{\cal V}}(p)-K_{{\cal V}}^{{\cal V}}(\Pi )\;\geq \;\tau _{%
{\cal V}}^{M_{1}}(p)-K_{{\cal V}}^{M_{1}}(\Pi )\;-\;\frac{r^{2}(r-2)}{2(r-1)}%
\,\Vert H\Vert ^{2}.  \label{eq-GCFI-(1)}
\end{equation}%
The equality holds in {\rm (\ref{eq-GCFI-(1)})} if and only if the matrix $(%
{\cal T}^{{\cal H}})_{ij}^{\ell }$, $i,j=1,\ldots ,r$, takes the following
forms: 
\begin{eqnarray}
\left( {\cal T}^{{\cal H}}\right) _{ij}^{1} &=&\left( 
\begin{array}{ccccc}
a & 0 & 0 & 0 & 0 \\ 
0 & b & 0 & 0 & 0 \\ 
0 & 0 & a+b & 0 & 0 \\ 
0 & 0 & 0 & \ddots  & 0 \\ 
0 & 0 & 0 & 0 & a+b%
\end{array}%
\right),   \label{eq-TH-matrix-1} \\
\left( {\cal T}^{{\cal H}}\right) _{ij}^{\ell } &=&\left( 
\begin{array}{ccccc}
a_{\ell } & b_{\ell } & 0 & \cdots  & 0 \\ 
b_{\ell } & -a_{\ell } & 0 & \cdots  & 0 \\ 
0 & 0 & 0 & \cdots  & 0 \\ 
0 & 0 & 0 & \ddots  & 0 \\ 
0 & 0 & 0 & \cdots  & 0%
\end{array}%
\right) ,\quad \ell =2,\cdots ,s,  \label{eq-TH-matrix-2}
\end{eqnarray}%
where $\left( {\cal T}^{{\cal H}}\right) _{11}^{1} \equiv a$, $\left( {\cal T%
}^{{\cal H}}\right) _{22}^{1}\equiv b$, $\left( {\cal T}^{{\cal H}}\right)
_{11}^{\ell }\equiv a_{\ell }$ and $\left( {\cal T}^{{\cal H}}\right)
_{12}^{\ell }\equiv b_{\ell }$.
\end{theorem}
\begin{proof}
By (\ref{eq-P-(10)}), (\ref{eq-P-(9.1)}) and (\ref%
{eq-P-(9.2.1)}), we obtain 
\[
2\tau _{{\cal V}}^{M_{1}}(p)=2\tau _{{\cal V}}^{{\cal V}}(p)+%
\sum_{i,j=1}^{r}\sum_{\ell =1}^{s}\left( \left( {\cal T}^{{\cal H}}\right)
_{ii}^{\ell }\left( {\cal T}^{{\cal H}}\right) _{jj}^{\ell }-\left( \left( 
{\cal T}^{{\cal H}}\right) _{ij}^{\ell }\right) ^{2}\right) . 
\]%
By using (\ref{eq-P-(9)}) and (\ref{eq-P-(14)}), we obtain 
\begin{equation}
2\tau _{{\cal V}}^{M_{1}}(p)=2\tau _{{\cal V}}^{{\cal V}}(p)+r^{2}\left\Vert
H\right\Vert ^{2}-\left\Vert {\cal T}^{{\cal H}}\right\Vert ^{2}.
\label{eq-(3)}
\end{equation}%
Define 
\begin{equation}
\varepsilon =2\tau _{{\cal V}}^{M_{1}}(p)-2\tau _{{\cal V}}^{{\cal V}}(p)-%
\frac{r^{2}\left( r-2\right) }{\left( r-1\right) }\left\Vert H\right\Vert
^{2}.  \label{eq-(4)}
\end{equation}%
From (\ref{eq-(3)}) and (\ref{eq-(4)}), it follows that 
\begin{equation}
r^{2}\Vert H\Vert ^{2}=(r-1)\left( \varepsilon +\Vert {\cal T}^{{\cal H}%
}\Vert ^{2}\right) .  \label{eq-(5)}
\end{equation}%
Since the mean curvature vector $H$ is in the direction of $U_{1}$, then by (%
\ref{eq-(5)}), we obtain 
\begin{eqnarray*}
\left( \sum_{i=1}^{r}\left( {\cal T}^{{\cal H}}\right) _{ii}^{1}\right) ^{2}
&=&\left( r-1\right) \left\{ \sum_{i=1}^{r}\left( \left( {\cal T}^{{\cal H}%
}\right) _{ii}^{1}\right) ^{2}+\sum_{i\neq j=1}^{r}\left( \left( {\cal T}^{%
{\cal H}}\right) _{ij}^{1}\right) ^{2}\right. \\
&&\qquad \qquad \left. +\sum_{i,j=1}^{r}\sum_{\ell =2}^{s}\left( \left( 
{\cal T}^{{\cal H}}\right) _{ij}^{\ell }\right) ^{2}+\varepsilon \right\} .
\end{eqnarray*}%
Applying Lemma~{\rm \ref{Lemma 2}}, we obtain 
\begin{equation}
2\left( {\cal T}^{{\cal H}}\right) _{11}^{1}\left( {\cal T}^{{\cal H}%
}\right) _{22}^{1}\geq \sum_{i\neq j=1}^{r}\left( \left( {\cal T}^{{\cal H}%
}\right) _{ij}^{1}\right) ^{2}+\sum_{i,j=1}^{r}\sum_{\ell =2}^{s}\left(
\left( {\cal T}^{{\cal H}}\right) _{ij}^{\ell }\right) ^{2}+\varepsilon .
\label{eq-(6)}
\end{equation}%
Now, if $F_{1}=F_{4}=V_{1}$ and $F_{2}=F_{3}=V_{2}$ in (\ref{eq-P-(10)}),
then we obtain 
\begin{eqnarray}
K_{{\cal V}}^{M_{1}}\left( \Pi \right) &=&K_{{\cal V}}^{{\cal V}}\left( \Pi
\right) +\left( {\cal T}^{{\cal H}}\right) _{11}^{1}\left( {\cal T}^{{\cal H}%
}\right) _{22}^{1}-\left( \left( {\cal T}^{{\cal H}}\right) _{12}^{1}\right)
^{2}  \nonumber \\
&&+\sum_{\ell =2}^{s}\left( \left( {\cal T}^{{\cal H}}\right) _{11}^{\ell
}\left( {\cal T}^{{\cal H}}\right) _{22}^{\ell }-\left( \left( {\cal T}^{%
{\cal H}}\right) _{12}^{\ell }\right) ^{2}\right) .  \label{eq-(6a)}
\end{eqnarray}%
From (\ref{eq-(6)}) and (\ref{eq-(6a)}), we get 
\begin{eqnarray*}
K_{{\cal V}}^{M_{1}}\left( \Pi \right) &\geq &K_{{\cal V}}^{{\cal V}}\left(
\Pi \right) +\sum_{\ell =1}^{s}\sum_{j>2}^{r}\left( \left( \left( {\cal T}^{%
{\cal H}}\right) _{1j}^{\ell }\right) ^{2}+\left( \left( {\cal T}^{{\cal H}%
}\right) _{2j}^{\ell }\right) ^{2}\right) \\
&&+\frac{1}{2}\sum_{i\neq j>2}^{r}\left( \left( \left( {\cal T}^{{\cal H}%
}\right) _{ij}^{1}\right) ^{2}\right) +\sum_{\ell
=2}^{s}\sum_{i,j>2}^{r}\left( \left( {\cal T}^{{\cal H}}\right) _{ij}^{\ell
}\right) ^{2} \\
&&+\sum_{\ell =2}^{s}\frac{1}{2}\left( \left( {\cal T}^{{\cal H}}\right)
_{11}^{\ell }+\left( {\cal T}^{{\cal H}}\right) _{22}^{\ell }\right) ^{2}+%
\frac{\varepsilon }{2}.
\end{eqnarray*}%
or 
\begin{equation}
K_{{\cal V}}^{{\cal V}}(\Pi )\leq K_{{\cal V}}^{M_{1}}(\Pi )-\frac{%
\varepsilon }{2}.  \label{eq-(7)}
\end{equation}%
In view of (\ref{eq-(4)}) and (\ref{eq-(7)}), we get (\ref{eq-GCFI-(1)}). If the equality in (\ref{eq-GCFI-(1)}) holds, then the inequalities given by
(\ref{eq-(7)}) and (\ref{eq-(6)}) become equalities. In this case, we have 
\begin{align*}
\left( {\cal T}^{{\cal H}}\right) _{1j}^{\ell }=\left( {\cal T}^{{\cal H}%
}\right) _{2j}^{\ell }& =0, & & j>2,\;\ell =1,\ldots ,s, \\
\left( {\cal T}^{{\cal H}}\right) _{ij}^{1}& =0, & & i\neq j>2, \\
\left( {\cal T}^{{\cal H}}\right) _{ij}^{\ell }& =0, & & i,j>2,\;\ell
=2,\ldots ,s, \\
\left( {\cal T}^{{\cal H}}\right) _{11}^{\ell }+\left( {\cal T}^{{\cal H}%
}\right) _{22}^{\ell }& =0, & & \ell =2,\ldots ,s.
\end{align*}
Moreover, we can choose an orthonormal basis $\{V_{1},V_{2}\}$ such that $%
\left( {\cal T}^{{\cal H}}\right) _{12}^{1}=\left( {\cal T}^{{\cal H}%
}\right) _{21}^{1}=0$. By applying Lemma~\ref{Lemma 2}, it follows that 
\[
\left( {\cal T}^{{\cal H}}\right) _{11}^{1}+\left( {\cal T}^{{\cal H}%
}\right) _{22}^{1}=\left( {\cal T}^{{\cal H}}\right) _{33}^{1}=\cdots
=\left( {\cal T}^{{\cal H}}\right) _{rr}^{1}. 
\]
\end{proof} 
\begin{theorem}
Let $\pi :(M_{1},g_{1})\rightarrow (M_{2},g_{2})$ be a Riemannian submersion
between two Riemannian manifolds with $\dim M_{1}=n$ and $\dim M_{2}=m$. If $%
\dim {\cal V}_{p}=r>2$, then for $2$-plane $\Pi \subset {\cal V}_{p}$, 
\begin{equation}
\hat{\delta}^{{\cal V}}\left( 2\right) \geq \;\tau _{{\cal V}%
}^{M_{1}}(p)-\sup \left\{ K_{{\cal V}}^{M_{1}}(\Pi )\right\} \;-\;\frac{%
r^{2}(r-2)}{2(r-1)}\,\Vert H\Vert ^{2}.  \label{eq-delta-1}
\end{equation}%
The equality holds in {\rm (\ref{eq-delta-1})} if and only if the tensor $%
{\cal T}^{{\cal H}}$ takes the form given by {\rm (\ref{eq-TH-matrix-1})}
and {\rm (\ref{eq-TH-matrix-2})}.
\end{theorem}

\begin{proof} From (\ref{eq-GCFI-(1)}), we have 
\begin{equation}
\tau _{{\cal V}}^{{\cal V}}(p)-K_{{\cal V}}^{{\cal V}}(\Pi )\geq \inf
\left\{ \tau _{{\cal V}}^{M_{1}}(p)-K_{{\cal V}}^{M_{1}}(\Pi )\;-\;\frac{%
r^{2}(r-2)}{2(r-1)}\,\Vert H\Vert ^{2}\right\} ,  \label{eq-delta-2}
\end{equation}%
that is $\tau _{{\cal V}}^{M_{1}}(p)-K_{{\cal V}}^{M_{1}}(\Pi )\;-\;\frac{%
r^{2}(r-2)}{2(r-1)}\,\Vert H\Vert ^{2}$ is a lower bound for $\tau _{{\cal V}%
}^{{\cal V}}(p)-K_{{\cal V}}^{{\cal V}}(\Pi )$. Hence, from (\ref{eq-delta-2}%
), we get 
\begin{equation}
\inf \left\{ \tau _{{\cal V}}^{{\cal V}}(p)-K_{{\cal V}}^{{\cal V}}(\Pi
)\right\} \geq \;\inf \left\{ \tau _{{\cal V}}^{M_{1}}(p)-K_{{\cal V}%
}^{M_{1}}(\Pi )\;-\;\frac{r^{2}(r-2)}{2(r-1)}\,\Vert H\Vert ^{2}\right\} .
\label{eq-delta-3}
\end{equation}%
Since for a fixed ${\cal V}_{p}$, $\tau _{{\cal V}}^{{\cal V}}(p)$, $\frac{%
r^{2}(r-2)}{2(r-1)}\,\Vert H\Vert ^{2}$ and $\tau _{{\cal V}}^{M_{1}}(p)$
are real constants. Thus, from (\ref{eq-delta-3}), we have 
\[
\tau _{{\cal V}}^{{\cal V}}(p)-\sup \left\{ K_{{\cal V}}^{{\cal V}}(\Pi
)\right\} \geq \;\tau _{{\cal V}}^{M_{1}}(p)-\sup \left\{ K_{{\cal V}%
}^{M_{1}}(\Pi )\right\} \;-\;\frac{r^{2}(r-2)}{2(r-1)}\,\Vert H\Vert ^{2}. 
\]%
In view of (\ref{eq-delta-V-02}), we get (\ref{eq-delta-1}). The equality
holds in {\rm (\ref{eq-delta-1})} if and only if the tensor ${\cal T}^{{\cal %
H}}$ takes the form given by {\rm (\ref{eq-TH-matrix-1})} and {\rm (\ref%
{eq-TH-matrix-2})}. \newline
\end{proof}
We construct the following two examples that satisfy the assumptions of
Theorem~\ref{Theorem 1}. Among these, one example attains the equality in
inequality obtained in Theorem~\ref{Theorem 1}, while the other does not
achieve equality.
\begin{example}
\label{Exa-GIRMEDNH} Let $\pi :{\Bbb R}^{6}\rightarrow {\Bbb R}^{3}$ be
defined by 
\[
\pi (x_{1},x_{2},\ldots ,x_{6})=\left( x_{3}\sin \alpha -x_{5}\cos \alpha
,\;x_{4},\;x_{6}\right) ,\qquad 0^{\circ }<\alpha <90^{\circ }. 
\]%
Let the Riemannian metrics $g_{1}$ and $g_{2}$ on the total and base spaces
be given by 
\[
g_{1}=e^{2x_{4}}(dx_{1})^{2}+e^{2x_{6}}(dx_{2})^{2}+(dx_{3})^{2}+e^{2x_{6}}(dx_{4})^{2}+(dx_{5})^{2}+e^{2x_{4}}(dx_{6})^{2}, 
\]%
and 
\[
g_{2}=(dy_{1})^{2}+e^{2y_{3}}(dy_{2})^{2}+e^{2y_{2}}(dy_{3})^{2}. 
\]%
Then the vertical and horizontal distributions are given by 
\[
{\cal V}=\ker \pi _{\ast }={\rm span}\left\{ e^{-x_{4}}\frac{\partial }{%
\partial x_{1}},\;e^{-x_{6}}\frac{\partial }{\partial x_{2}},\;\cos \alpha 
\frac{\partial }{\partial x_{3}}+\sin \alpha \frac{\partial }{\partial x_{5}}%
\right\} , 
\]%
\[
{\cal H}={\rm span}\left\{ U_{1}=\sin \alpha \frac{\partial }{\partial x_{3}}%
-\cos \alpha \frac{\partial }{\partial x_{5}},\;U_{2}=e^{-x_{6}}\frac{%
\partial }{\partial x_{4}},\;U_{3}=e^{-x_{4}}\frac{\partial }{\partial x_{6}}%
\right\} . 
\]%
A direct computation yields 
\[
\pi _{\ast }U_{1}=\frac{\partial }{\partial y_{1}},\quad \pi _{\ast
}U_{2}=e^{-y_{3}}\frac{\partial }{\partial y_{2}},\quad \pi _{\ast
}U_{3}=e^{-y_{2}}\frac{\partial }{\partial y_{3}}. 
\]%
Now we verify the isometry condition: 
\[
g_{1}(U_{1},U_{1})=1=g_{2}\left( \frac{\partial }{\partial y_{1}},\frac{%
\partial }{\partial y_{1}}\right) , 
\]%
\[
g_{1}(U_{2},U_{2})=e^{2x_{6}}\cdot e^{-2x_{6}}=1=g_{2}\left( e^{-y_{3}}\frac{%
\partial }{\partial y_{2}},e^{-y_{3}}\frac{\partial }{\partial y_{2}}\right)
, 
\]%
\[
g_{1}(U_{3},U_{3})=e^{2x_{4}}\cdot e^{-2x_{4}}=1=g_{2}\left( e^{-y_{2}}\frac{%
\partial }{\partial y_{3}},e^{-y_{2}}\frac{\partial }{\partial y_{3}}\right)
, 
\]%
and $g_{1}(U_{i},U_{j})=0=g_{2}(\pi _{\ast }U_{i},\pi _{\ast }U_{j})$ for $%
i\neq j$. Thus, 
\[
g_{1}(U_{i},U_{j})=g_{2}(\pi _{\ast }U_{i},\pi _{\ast }U_{j}),\quad i,j\in
\{1,2,3\}, 
\]%
and ${\rm rank}\,\pi =3$. Hence, $\pi $ is a Riemannian submersion. Next,
using the Christoffel symbol formula 
\[
\Gamma _{jk}^{i}=\frac{1}{2}g^{ii}\left( \frac{\partial g_{ik}}{\partial
x_{j}}+\frac{\partial g_{ij}}{\partial x_{k}}-\frac{\partial g_{jk}}{%
\partial x_{i}}\right) , 
\]%
the non-zero derivatives of metric components are 
\[
\frac{\partial g_{11}}{\partial x_{4}} =2e^{2x_{4}},\quad \frac{\partial
g_{22}}{\partial x_{6}}=2e^{2x_{6}}, \quad \frac{\partial g_{44}}{\partial x_{6}} =2e^{2x_{6}},\quad \frac{\partial g_{66}}{\partial x_{4}}=2e^{2x_{4}}.
\]
From this, one obtains the non-vanishing Christoffel symbols: 
\[
\Gamma _{14}^{1}=\Gamma _{41}^{1}=1,\quad \Gamma _{26}^{2}=\Gamma
_{62}^{2}=1, 
\]%
\[
\Gamma _{11}^{4}=-e^{2(x_{4}-x_{6})},\quad \Gamma _{46}^{4}=\Gamma
_{64}^{4}=1,\quad \Gamma _{66}^{4}=-e^{2(x_{4}-x_{6})}, 
\]%
\[
\Gamma _{22}^{6}=-e^{2(x_{6}-x_{4})},\quad \Gamma
_{44}^{6}=-e^{2(x_{6}-x_{4})},\quad \Gamma _{46}^{6}=\Gamma _{64}^{6}=1. 
\]%
The covariant derivatives of vertical vector fields are 
\begin{align*}
\nabla _{V_{1}}V_{1}& =\nabla _{\displaystyle\frac{1}{e^{x_{4}}}\frac{%
\partial }{\partial x_{1}}}\left( \frac{1}{e^{x_{4}}}\frac{\partial }{%
\partial x_{1}}\right) =-e^{-2x_{6}}\frac{\partial }{\partial x_{4}}, \\
\nabla _{V_{2}}V_{2}& =\nabla _{\displaystyle\frac{1}{e^{x_{6}}}\frac{%
\partial }{\partial x_{2}}}\left( \frac{1}{e^{x_{6}}}\frac{\partial }{%
\partial x_{2}}\right) =-e^{-2x_{4}}\frac{\partial }{\partial x_{6}}, \\
\nabla _{V_{3}}V_{3}& =\nabla _{\displaystyle\cos \alpha \frac{\partial }{%
\partial x_{3}}+\sin \alpha \frac{\partial }{\partial x_{5}}}\left( \cos
\alpha \frac{\partial }{\partial x_{3}}+\sin \alpha \frac{\partial }{%
\partial x_{5}}\right) =0,
\end{align*}%
and%
\[
\nabla _{V_{i}}V_{j}=0,\quad i\neq j. 
\]%
Components of the $T^{{\cal H}}$ tensor are%
\begin{eqnarray*}
\left( {\cal T}^{{\cal H}}\right) _{11}^{1} &=&0,\ \left( {\cal T}^{{\cal H}%
}\right) _{11}^{2}=-e^{-3x_{6}},\quad \left( {\cal T}^{{\cal H}}\right)
_{11}^{3}=0, \\
\left( {\cal T}^{{\cal H}}\right) _{22}^{1} &=&0,\ \left( {\cal T}^{{\cal H}%
}\right) _{22}^{2}=0,\quad \left( {\cal T}^{{\cal H}}\right)
_{22}^{3}=-e^{-3x_{4}}, \\
\left( {\cal T}^{{\cal H}}\right) _{33}^{1} &=&0,\ \left( {\cal T}^{{\cal H}%
}\right) _{33}^{2}=0,\quad \left( {\cal T}^{{\cal H}}\right) _{33}^{3}=0, \\
\left( {\cal T}^{{\cal H}}\right) _{ij}^{\alpha } &=&0,\ i\neq j\quad \alpha
\in \left\{ 1,2,3\right\} .
\end{eqnarray*}%
We observe that%
\[
\left( {\cal T}^{{\cal H}}\right) _{11}^{2}+\left( {\cal T}^{{\cal H}%
}\right) _{22}^{2}=-e^{-3x_{6}}\neq 0. 
\]%
Clearly, for $\pi $ the inequality obtained in Theorem \ref{Theorem 1} does
not attain equality.
\end{example}
\begin{example}
\label{Exa-GIGSEH}Let 
\[
N_{1}=\left\{ \left( x_{1},x_{2},x_{3},x_{4},x_{5},x_{6}\right) \in {\Bbb R}%
^{6}\ :\ x_{1},x_{3},x_{5}>0\right\} 
\]
and 
\[
N_{2}=\left\{ \left( y_{1},y_{2},y_{3}\right) \in {\Bbb R}^{3}\right\} 
\]
be two Riemannian manifolds equipped with the Riemannian metrics 
\begin{align*}
g_{1} &= (x_{1})^{2}(dx_{1})^{2} + (dx_{2})^{2} + (x_{3})^{2}(dx_{3})^{2} \\
&\quad + (dx_{4})^{2} + (x_{5})^{2}(dx_{5})^{2} + (dx_{6})^{2}, \\
g_{2} &= (dy_{1})^{2} + (dy_{2})^{2} + (dy_{3})^{2},
\end{align*}
respectively. Define a map 
\[
\pi : (N_{1}, g_{1}) \longrightarrow (N_{2}, g_{2}) 
\]
by 
\[
\pi(x_{1},x_{2},x_{3},x_{4},x_{5},x_{6})=\left( x_{2},x_{4},x_{6}\right) . 
\]%
Then we have%
\begin{eqnarray*}
{\cal V} &=&{\rm span}\left\{ V_{1}=e_{1},V_{2}=e_{3},V_{3}=e_{5}\right\}, \\
{\cal H} &=&{\rm span}\left\{ U_{1}=e_{2},U_{2}=e_{4},U_{3}=e_{6}\right\}, \\
{\rm range}\pi _{\ast } &=&{\rm span}\left\{ \pi _{\ast }U_{1}=e_{1}^{\ast
},\pi _{\ast }U_{2}=e_{2}^{\ast },\pi _{\ast }U_{3}=e_{3}^{\ast }\right\} ,
\end{eqnarray*}%
where $\left\{ e_{1}=\frac{1}{x_{1}}\frac{\partial }{\partial x_{1}},e_{2}=%
\frac{\partial }{\partial x_{2}},e_{3}=\frac{1}{x_{3}}\frac{\partial }{%
\partial x_{3}},e_{4}=\frac{\partial }{\partial x_{4}},e_{5}=\frac{1}{x_{5}}%
\frac{\partial }{\partial x_{5}},e_{6}=\frac{\partial }{\partial x_{6}}%
\right\} $ and \newline
$\left\{ e_{1}^{\ast }=\frac{\partial }{\partial y_{1}},e_{2}^{\ast }=\frac{%
\partial }{\partial y_{2}},e_{3}^{\ast }=\frac{\partial }{\partial y_{3}}%
\right\} $ be orthonormal bases of $T_{p}M$ and $T_{\pi \left( p\right) }N$,
respectively. We observe that $g_{1}\left( U_{i},U_{j}\right) =g_{2}\left(
\pi _{\ast }U_{i},\pi _{\ast }U_{j}\right) $ for all $U_{i},U_{j}\in {\cal H}
$ and ${\rm rank}~\pi =3$. Thus $\pi $ is a Riemannian submersion. The
non-zero Christoffel symbols of $g_{1}$ are%
\[
\Gamma _{ii}^{i}=\frac{1}{x_{i}},\quad i=1,3,5. 
\]%
Covariant derivatives of the vertical vector fields are 
\begin{eqnarray*}
\nabla _{V_{1}}V_{1} &=&\nabla _{\displaystyle\frac{1}{x_{1}}\frac{\partial 
}{\partial x_{1}}}\frac{1}{x_{1}}\frac{\partial }{\partial x_{1}}=\frac{1}{%
x_{1}}\left( -\frac{1}{x_{1}^{2}}\frac{\partial }{\partial x_{1}}+\frac{1}{%
x_{1}}\Gamma _{11}^{1}\frac{\partial }{\partial x_{1}}\right) =0, \\
\nabla _{V_{2}}V_{2} &=&\nabla _{\displaystyle\frac{1}{x_{3}}\frac{\partial 
}{\partial x_{3}}}\frac{1}{x_{3}}\frac{\partial }{\partial x_{3}}=\frac{1}{%
x_{3}}\left( -\frac{1}{x_{3}^{2}}\frac{\partial }{\partial x_{3}}+\frac{1}{%
x_{3}}\Gamma _{33}^{3}\frac{\partial }{\partial x_{3}}\right) =0, \\
\nabla _{V_{3}}V_{3} &=&\nabla _{\displaystyle\frac{1}{x_{5}}\frac{\partial 
}{\partial x_{5}}}\frac{1}{x_{5}}\frac{\partial }{\partial x_{5}}=\frac{1}{%
x_{5}}\left( -\frac{1}{x_{5}^{2}}\frac{\partial }{\partial x_{5}}+\frac{1}{%
x_{5}}\Gamma _{55}^{5}\frac{\partial }{\partial x_{5}}\right) =0,
\end{eqnarray*}%
and%
\[
\nabla _{V_{i}}V_{j}=0,\quad i\neq j. 
\]%
Then, we obtain 
\[
\left( {\cal T}^{{\cal H}}\right) _{ij}^{\alpha }=0,\quad 1\leq i,j\leq
3,\quad 1\leq \alpha \leq 3. 
\]%
Hence, we observe that the inequality obtained in {\rm Theorem \ref{Theorem
1}} attains equality.
\end{example}
\begin{theorem}
\label{Theorem RSFCF}Let $\pi :\left( {M_{1}}(c),g_{1}\right) \rightarrow
\left( M_{2},g_{2}\right) $ be a Riemannian submersion from real space form
of constant curvature $c$ onto a Riemannian manifold with $\dim {\cal V}%
_{p}=r>2$, then for any $2$-plane $\Pi \subset {\cal V}_{p}$, 
\begin{equation}
\tau _{{\cal V}}^{{\cal V}}(p)-K_{{\cal V}}^{{\cal V}}(\Pi )\geq \frac{1}{2}%
\left\{ c\left( r^{2}-r-2\right) -\frac{r^{2}\left( r-2\right) }{\left(
r-1\right) }\left\Vert H\right\Vert ^{2}\right\} .  \label{eq-RSF-(1)}
\end{equation}%
The equality holds in {\rm (\ref{eq-RSF-(1)}) }if and only if the tensor $%
{\cal T}^{{\cal H}}$ takes the form given by {\rm (\ref{eq-TH-matrix-1})}
and {\rm (\ref{eq-TH-matrix-2})}.
\end{theorem}
\begin{proof} By (\ref{eq-P-(9.2)}), (\ref{eq-P-(9.2.1)}) and (\ref%
{eq-RSF}), we get 
\begin{equation}
\tau _{{\cal V}}^{M_{1}}(p)=\frac{r\left( r-1\right) }{2}c,\quad K_{{\cal V}%
}^{M_{1}}\left( \Pi \right) =c.  \label{eq-RSF-(1.1)}
\end{equation}%
In view of (\ref{eq-GCFI-(1)}) and (\ref{eq-RSF-(1.1)}), we get (\ref%
{eq-RSF-(1)}). \end{proof}
\begin{example}
{\rm \cite{Lee_Lee_Sahin_Vilcu_21_AMPA}} \label{Exa-RSFEH}
Consider the standard Riemannian submersion
\[
\pi :{\Bbb S}^{15}(1)\rightarrow {\Bbb S}^{8}\left( \frac{1}{2}\right),
\]
whose fibers are totally geodesic and isometric to ${\Bbb S}^{7}$. Here, ${\Bbb S}^{15}(1)$ denotes the $15$-dimensional unit sphere of constant sectional curvature $1$, whereas ${\Bbb S}^{8}\left( \frac{1}{2}\right)$ is the $8$-dimensional sphere of constant sectional curvature $4$. This submersion gives an example for which equality holds in the inequality obtained in Theorem~\ref{Theorem RSFCF}.
\end{example}
\begin{theorem}
\label{Theorem CSFCF}Let $\pi :\left( {M_{1}}(c),g_{1}\right) \rightarrow
\left( M_{2},g_{2}\right) $ be a Riemannian submersion\ from a complex space
form of constant holomorphic sectional curvature $c$ onto a Riemannian
manifold with $\dim M_{1}=n=2k$, $\dim M_{2}=m$ and $\dim {\cal V}_{p}=r>2$,
then for any $2$-plane $\Pi \subset {\cal V}_{p}$, 
\begin{eqnarray}
\tau _{{\cal V}}^{{\cal V}}(p)-K_{{\cal V}}^{{\cal V}}(\Pi ) &\geq &\frac{1}{%
2}\left\{ \frac{c}{4}\left( r^{2}-r-2\right) +\frac{3c}{4}\left( \left\Vert
Q\right\Vert ^{2}-2\left( g_{1}\left( V_{1},QV_{2}\right) \right)
^{2}\right) \right.  \nonumber \\
&&\ \ \ \ \ \left. -\frac{r^{2}\left( r-2\right) }{\left( r-1\right) }%
\left\Vert H\right\Vert ^{2}\right\} .  \label{eq-GCSFMI}
\end{eqnarray}%
The equality holds in {\rm (\ref{eq-GCSFMI}) }if and only if the tensor $%
{\cal T}^{{\cal H}}$ takes the form given by {\rm (\ref{eq-TH-matrix-1})}
and {\rm (\ref{eq-TH-matrix-2})}.
\end{theorem}
\begin{proof}
By (\ref{eq-P-(9.2)}), (\ref{eq-P-(9.2.1)}) and (\ref{eq-GCSF}), we get 
\begin{equation}
\tau _{{\cal V}}^{M_{1}}(p)=\frac{r\left( r-1\right) }{2}\frac{c}{4}+\frac{3c%
}{8}\left\Vert Q\right\Vert ^{2},\quad K_{{\cal V}}^{M_{1}}\left( \Pi
\right) =\frac{c}{4}+\frac{3c}{4}\left( g_{1}\left( V_{1},QV_{2}\right)
\right) ^{2}.  \label{eq-GCSF-(1.1)}
\end{equation}
In view of (\ref{eq-GCFI-(1)}) and (\ref{eq-GCSF-(1.1)}), we get (\ref%
{eq-GCSFMI}).
\end{proof}
\begin{example}
\label{Exa-CSFEH} For the Riemannian submersion, example constructed in {\rm 
\cite[~Example ~31]{Singh_Meena_Meena_26_JMAA}}, we define compatible almost
complex structure $J$ on ${\Bbb R}^{6}$ as 
\[
J\left( x_{1},x_{2},x_{3},x_{4},x_{5},x_{6}\right) =\left(
-x_{2},x_{1},-x_{4},x_{3},-x_{6},x_{5}\right) . 
\]%
We observe that $({\Bbb R}^{6},g_{1},J)$ becomes a complex space form, and we
get 
\[
\left( T^{{\cal H}}\right) _{11}^{1}+\left( T^{{\cal H}}\right)
_{22}^{1}\neq \left( T^{{\cal H}}\right) _{33}^{1}. 
\]
Clearly, we observe that the inequality obtained in {\rm Theorem~\ref%
{Theorem CSFCF}} does not attain equality.
\end{example}
\begin{theorem}
\label{Theorem GSSFCF}Let $\pi :\left( {M_{1}}(c_{1},c_{2},c_{3}),g_{1}%
\right) \rightarrow \left( M_{2},g_{2}\right) $ be a Riemannian submersion
from a generalized Sasakian space form onto a Riemannian manifold with $\dim
M_{1}=n=2k+1$, $\dim M_{2}=m$ and $\dim {\cal V}_{p}=r>2$, then for any $2$%
-plane $\Pi \subset {\cal V}_{p}$,
\begin{enumerate}
\item[{\bf (i)}] if $\xi \in {\cal V}_{p}$, 
\begin{eqnarray}
\tau _{{\cal V}}^{{\cal V}}(p)-K_{{\cal V}}^{{\cal V}}(\Pi ) &\geq &\frac{%
c_{1}}{2}\left( r^{2}-r-2\right) +\frac{3}{2}c_{2}\left( \left\Vert
Q\right\Vert ^{2}-2\left( g_{1}\left( V_{1},QV_{2}\right) \right) ^{2}\right)
\nonumber \\
&&-c_{3}\left( \left( r-1\right) -\Theta (\Pi )\right) -\frac{r^{2}\left(
r-2\right) }{2\left( r-1\right) }\left\Vert H\right\Vert ^{2},
\label{eq-GSSFMI-1}
\end{eqnarray}

\item[{\bf (ii)}] if $\xi \in {\cal H}_{p}$,%
\begin{eqnarray}
\tau _{{\cal V}}^{{\cal V}}(p)-K_{{\cal V}}^{{\cal V}}(\Pi ) &\geq &\frac{%
c_{1}}{2}\left( r^{2}-r-2\right) +\frac{3}{2}c_{2}\left( \left\Vert
Q\right\Vert ^{2}-2\left( g_{1}\left( V_{1},QV_{2}\right) \right) ^{2}\right)
\nonumber \\
&&-\frac{r^{2}\left( r-2\right) }{2\left( r-1\right) }\left\Vert
H\right\Vert ^{2},  \label{eq-GSSFMI-2}
\end{eqnarray}%
where $\Theta (\Pi )=\left( \left( \eta \left( V_{1}\right) \right)
^{2}+\left( \eta \left( V_{2}\right) \right) ^{2}\right) $. The equality
holds in {\rm (\ref{eq-GSSFMI-1})} and {\rm (\ref{eq-GSSFMI-2})} if and only
if the tensor ${\cal T}^{{\cal H}}$ takes the form given by {\rm (\ref%
{eq-TH-matrix-1})} and {\rm (\ref{eq-TH-matrix-2})}.
\end{enumerate}
\end{theorem}
\begin{proof}
By (\ref{eq-P-(9.2)}), (\ref{eq-P-(9.2.1)}) and (\ref%
{eq-GSSF}), we get 
\begin{equation}
\tau _{{\cal V}}^{M_{1}}(p)=\left\{ 
\begin{array}{ll}
\frac{r\left( r-1\right) }{2}c_{1}+\frac{3}{2}c_{2}\left\Vert Q\right\Vert
^{2}-\left( r-1\right) c_{3} & {\rm if}~\xi \in {\cal V}_{p}; \\ 
\frac{r\left( r-1\right) }{2}c_{1}+\frac{3}{2}c_{2}\left\Vert Q\right\Vert
^{2} & {\rm if}~\xi \in {\rm {\cal H}}_{p},%
\end{array}%
\right.  \label{eq-GSSFMI-(1.1)}
\end{equation}%
and 
\begin{equation}
K_{{\cal V}}^{M_{1}}\left( \Pi \right) =\left\{ 
\begin{array}{ll}
\begin{array}{l}
c_{1}+3c_{2}\left( g_{1}\left( V_{1},QV_{2}\right) \right) ^{2} \\ 
-c_{3}\left( \left( \eta \left( V_{1}\right) \right) ^{2}+\left( \eta \left(
V_{2}\right) \right) ^{2}\right)%
\end{array}
& {\rm if}~\xi \in {\cal V}_{p}; \\ 
c_{1}+3c_{2}\left( g_{1}\left( V_{1},QV_{2}\right) \right) ^{2} & {\rm if}%
~\xi \in {\rm {\cal H}}_{p}.%
\end{array}%
\right.  \label{eq-GSSFMI-(1.2)}
\end{equation}%
In view of (\ref{eq-GCFI-(1)}), (\ref{eq-GSSFMI-(1.1)}) and (\ref%
{eq-GSSFMI-(1.2)}), we get (\ref{eq-GSSFMI-1}) and (\ref{eq-GSSFMI-2}). 
\end{proof}
\begin{example}
\label{Exa-GSSF} For the Riemannian submersion constructed in {\rm \cite[%
~Example ~32]{Singh_Meena_Meena_26_JMAA}}, it can be verified that the
inequality stated in {\rm Theorem~\ref{Theorem GSSFCF}} is satisfied with
equality.
\end{example}
\section{B.-Y. Chen inequality for Riemannian submersion along horizontal
distributions \label{section_Hor}}
\begin{theorem}
\label{Theorem-Horizontal Space}Let $\pi :(M_{1},g_{1})\rightarrow
(M_{2},g_{2})$ be a Riemannian submersion between two Riemannian manifolds
with $\dim M_{1}=n$ and $\dim M_{2}=m$. If $\dim {\cal H}_{p}=s>2$, then for
any $2$-plane ${\Bbb P}\subset {\cal H}_{p}$,%
\begin{equation}
\tau _{{\cal H}}^{{\cal H}}(p)-K_{{\cal H}}^{{\cal H}}({\Bbb P})\leq \tau _{%
{\cal H}}^{M_{1}}(p)-K_{{\cal H}}^{M_{1}}({\Bbb P}).
\label{eq-gen-Hori-Chen-First}
\end{equation}%
The equality holds in {\rm (\ref{eq-gen-Hori-Chen-First})} if and only if%
\begin{equation}
A_{1j}^{{\cal V}^{\alpha }}=0\quad {\rm for}\quad j\in \left\{ 3,\ldots
,s\right\} ,\quad \alpha \in \left\{ 1,\ldots ,r\right\} ,
\label{eq-hor-equality-1}
\end{equation}%
\begin{equation}
A_{ij}^{{\cal V}^{\alpha }}\quad {\rm for}\quad i,j\in \left\{ 2,\ldots
,s\right\} ,\quad \alpha \in \left\{ 1,\ldots ,r\right\} .
\label{eq-hor-equality-2}
\end{equation}
\end{theorem}
\begin{proof}
By (\ref{eq-P-(11)}), (\ref{eq-P-(9.1)}) and (\ref{eq-P-(9.2.1)}), we obtain 
\[
2\tau _{{\cal H}}^{{\cal H}}(p)=2\tau _{{\cal H}}^{M_{1}}(p)-3\sum_{\alpha
=1}^{r}\sum_{i,j=1}^{s}\left( A_{ij}^{{\cal V}^{\alpha }}\right) ^{2}, 
\]%
consequently, it follows that 
\begin{equation}
\tau _{{\cal H}}^{{\cal H}}(p)=\tau _{{\cal H}}^{M_{1}}(p)-\frac{3}{2}%
\left\Vert A^{{\cal V}}\right\Vert ^{2}.  \label{eq-scal-hori-Gauss}
\end{equation}%
On the other hand, from (\ref{eq-P-(9.2.0)}) and (\ref{eq-P-(14.1)}), we
obtain 
\begin{eqnarray}
\frac{3}{2}\left\Vert A^{{\cal V}}\right\Vert ^{2} &=&K_{{\cal H}}^{M_{1}}(%
{\Bbb P})-K_{{\cal H}}^{{\cal H}}({\Bbb P})+3\sum_{j=3}^{s}\sum_{\alpha
=1}^{r}\left( \left( A^{{\cal V}}\right) _{1j}^{\alpha }\right) ^{2}+\frac{3}{2}\sum_{i,j=2}^{s}\sum_{\alpha =1}^{r}\left( \left( A^{{\cal V}%
}\right) _{ij}^{\alpha }\right) ^{2},  \label{eq-3/2 AVsquare}
\end{eqnarray}%
Substituting the values from (\ref{eq-3/2 AVsquare}) into (\ref%
{eq-scal-hori-Gauss}), we deduce 
\begin{eqnarray}
\tau _{{\cal H}}^{{\cal H}}\left( p\right) -K_{{\cal H}}^{{\cal H}}({\Bbb P}%
) &=&\tau _{{\cal H}}^{M_{1}}-K_{{\cal H}}^{M_{1}}({\Bbb P}%
)-3\sum_{j=3}^{s}\sum_{\alpha =1}^{r}\left( A_{1j}^{{\cal V}^{\alpha
}}\right) ^{2} -\frac{3}{2}\sum_{i,j=2}^{s}\sum_{\alpha =1}^{r}\left( A_{ij}^{%
{\cal V}^{\alpha }}\right) ^{2}.  \label{eq-rel-scalH-scalM}
\end{eqnarray}%
From (\ref{eq-rel-scalH-scalM}), it follows immediately that 
\begin{equation}
\tau _{{\cal H}}^{{\cal H}}\left( p\right) -K_{{\cal H}}^{{\cal H}}({\Bbb P}%
)\leq \tau _{{\cal H}}^{M_{1}}(p)-K_{{\cal H}}^{M_{1}}({\Bbb P}).
\label{eq-Ine-ScalH-scalM}
\end{equation}%
The equality in (\ref{eq-Ine-ScalH-scalM}) holds if and only if 
\[
A_{1j}^{{\cal V}^{\alpha }}=0\quad {\rm for}\quad j\in \left\{ 3,\ldots
,s\right\} ,\quad \alpha \in \left\{ 1,\ldots ,r\right\} , 
\]%
\[
A_{ij}^{{\cal V}^{\alpha }}=0\quad {\rm for}\quad i,j\in \left\{ 2,\ldots
,s\right\} ,\quad \alpha \in \left\{ 1,\ldots ,r\right\} . 
\]%
This completes the proof.
\end{proof}
\begin{theorem}
Let $\pi :(M_{1},g_{1})\rightarrow (M_{2},g_{2})$ be a Riemannian submersion
between two Riemannian manifolds with $\dim M_{1}=n$ and $\dim M_{2}=m$. If $%
s=\dim {\cal H}_{p}>2$, then for $2$-plane ${\Bbb P}\subset {\cal H}_{p}$, 
\begin{equation}
\delta ^{{\cal H}}(2)\leq \;\tau _{{\cal H}}^{M_{1}}(p)-\inf \left\{ K_{%
{\cal H}}^{M_{1}}({\Bbb P})\right\} .  \label{eq-delta-H-1}
\end{equation}%
The equality holds in {\rm (\ref{eq-delta-H-1})} if and only if the tensor $%
A^{{\cal V}}$ satisfies {\rm (\ref{eq-hor-equality-1})} and {\rm (\ref%
{eq-hor-equality-2})}.
\end{theorem}
\begin{proof}
From (\ref{eq-gen-Hori-Chen-First}), we have 
\begin{equation}
\tau _{{\cal H}}^{{\cal H}}(p)-K_{{\cal H}}^{{\cal H}}({\Bbb P})\leq \sup
\left\{ \tau _{{\cal H}}^{M_{1}}(p)-K_{{\cal H}}^{M_{1}}({\Bbb P})\right\} ,
\label{eq-delta-H-2}
\end{equation}%
that is $\tau _{{\cal H}}^{M_{1}}(p)-K_{{\cal H}}^{M_{1}}({\Bbb P})$ is a
upper bound for $\tau _{{\cal H}}^{{\cal H}}(p)-K_{{\cal H}}^{{\cal H}}(%
{\Bbb P})$. Hence, from (\ref{eq-delta-H-2}), we get 
\begin{equation}
\sup \left\{ \tau _{{\cal H}}^{{\cal H}}(p)-K_{{\cal H}}^{{\cal H}}({\Bbb P}%
)\right\} \leq \;\sup \left\{ \tau _{{\cal H}}^{M_{1}}(p)-K_{{\cal H}%
}^{M_{1}}({\Bbb P})\right\} .  \label{eq-delta-H-3}
\end{equation}%
Since for a fixed ${\cal H}_{p}$, $\tau _{{\cal H}}^{{\cal H}}(p)$ and $\tau
_{{\cal H}}^{M_{1}}(p)$ are real constants. Thus, from (\ref{eq-delta-H-3}),
we have 
\[
\tau _{{\cal H}}^{{\cal H}}(p)-\inf \left\{ K_{{\cal H}}^{{\cal H}}({\Bbb P}%
)\right\} \leq \;\tau _{{\cal H}}^{M_{1}}(p)-\inf \left\{ K_{{\cal H}%
}^{M_{1}}({\Bbb P})\right\} . 
\]%
In view of (\ref{eq-delta-H-02}), we get (\ref{eq-delta-H-1}). The equality
case follows Theorem {\rm \ref{Theorem-Horizontal Space}}.
\end{proof}
\begin{example}
For the Riemannian submersion, the example constructed in {\rm \cite[%
~Example ~31]{Singh_Meena_Meena_26_JMAA}}, we observe that the inequality
obtained in Theorem {\ref{Theorem-Horizontal Space}} attains equality.
\end{example}
Now, we apply {\rm Theorem \ref{Theorem-Horizontal Space}} to various cases
of Riemannian submersions, as given below.
\begin{theorem}
\label{Th_real_hor} Let $\pi :(M_{1},g_{1})\rightarrow (M_{2},g_{2})$ be a
Riemannian submersion real space form of constant sectional curvature $c$
onto a Riemannian manifold with $\dim M_{1}=n$ and $\dim M_{2}=m$. If $\dim 
{\cal H}_{p}=s>2$, then for any $2$-plane ${\Bbb P}\subset {\cal H}_{p}$, 
\begin{equation}
\tau _{{\cal H}}^{{\cal H}}\left( p\right) -K_{{\cal H}}^{{\cal H}}({\Bbb P}%
)\leq \frac{1}{2}\left( s^{2}-s-2\right) c.  \label{eq-hor-cf-real}
\end{equation}%
The equality holds in {\rm (\ref{eq-hor-cf-real})} if and only if the tensor 
$A^{{\cal V}}$ satisfies {\rm (\ref{eq-hor-equality-1})} and {\rm (\ref%
{eq-hor-equality-2})}.
\end{theorem}
\begin{proof} Using (\ref{eq-P-(9.2.0)}), (\ref{eq-P-(9.2.1)}) and (%
\ref{eq-RSF}), we obtain 
\begin{equation}
\tau^{M_{1}}_{{\cal H}}(p) = \frac{s(s-1)}{2}\,c, \qquad K^{M_{1}}_{{\cal H}%
}({\Bbb P}) = c.  \label{eq-scal-real}
\end{equation}
Substituting the values from (\ref{eq-scal-real}) into (\ref%
{eq-Ine-ScalH-scalM}), we directly obtain (\ref{eq-hor-cf-real}). 
\end{proof}
\begin{example}
Consider the Riemannian submersion constructed in {\rm \cite[Example~31]%
{Singh_Meena_Meena_26_JMAA}}. It can be verified that this example satisfies
all the assumptions of {\rm Theorem~\ref{Th_real_hor}}. Moreover, the
inequality established in this theorem is attained as an equality for this
particular case.
\end{example}
\begin{theorem}
\label{Th_complex_hor} Let $\pi :\left( {M_{1}}(c),g_{1}\right) \rightarrow
\left( M_{2},g_{2}\right) $ be a Riemannian submersion\ from complex space
form of constant holomorphic sectional curvature $c$ onto a Riemannian
manifold with$\dim {\cal H}_{p}=s>2$, then for any $2$-plane ${\Bbb P}%
\subset {\cal H}_{p}$, 
\begin{eqnarray}
\tau _{{\cal H}}^{{\cal H}}\left( p\right) -K_{{\cal H}}^{{\cal H}}({\Bbb P}%
) &\leq &\frac{1}{2}\left( s^{2}-s-2\right) c+\frac{3c}{8}\left( \left\Vert P\right\Vert ^{2}-2\left( g_{1}\left(
U_{1},PU_{2}\right) \right) ^{2}\right) .  \label{eq-hor-cf-complex}
\end{eqnarray}%
The equality holds in {\rm (\ref{eq-hor-cf-complex})} if and only if the
tensor $A^{{\cal V}}$ satisfies {\rm (\ref{eq-hor-equality-1})} and {\rm (%
\ref{eq-hor-equality-2})}.
\end{theorem}
\begin{proof} 
Using (\ref{eq-GCSF}), (\ref{eq-P-(9.2.1)}) and (\ref%
{eq-P-(9.2.0)}), we obtain 
\begin{equation}
\tau _{{\cal H}}^{M_{1}}(p)=\frac{s\left( s-1\right) }{2}\frac{c}{4}+\frac{3c%
}{8}\left\Vert P\right\Vert ^{2},\quad K_{{\cal H}}^{M_{1}}({\Bbb P})=\frac{c%
}{4}+\frac{3c}{4}\left( g_{1}\left( U_{1},PU_{2}\right) \right) ^{2}.
\label{eq-scal-complex}
\end{equation}%
In view of (\ref{eq-Ine-ScalH-scalM}) and (\ref{eq-scal-complex}), we get (%
\ref{eq-hor-cf-complex}). Substituting (\ref{eq-scal-complex}) into (\ref%
{eq-Ine-ScalH-scalM}), we directly obtain (\ref{eq-hor-cf-complex}). 
\end{proof}
\begin{example}
Consider the Riemannian submersion constructed in {\rm \cite[Example~31]%
{Singh_Meena_Meena_26_JMAA}}. Define a compatible almost complex structure $%
J $ on ${\Bbb R}^{6}$ by 
\[
J\left( x_{1},x_{2},x_{3},x_{4},x_{5},x_{6}\right)
=\left(-x_{2},x_{1},-x_{4},x_{3},-x_{6},x_{5}\right). 
\]
Then $({\Bbb R}^{6}, g_{1}, J)$ becomes a complex space form. It can be
verified that this example satisfies all the assumptions of {\rm Theorem~\ref%
{Th_complex_hor}}. Furthermore, the inequality established in this theorem
is attained as an equality for this case.
\end{example}

\begin{theorem}
\label{Th_GSSF_hor} Let $\pi :\left( {M_{1}}(c_{1},c_{2},c_{3}),g_{1}\right)
\rightarrow \left( M_{2},g_{2}\right) $ be a Riemannian submersion from a
generalized Sasakian space form onto a Riemannian manifold with $\dim {\cal H%
}_{p}=s>2$, then for any $2$-plane ${\Bbb P}\subset {\cal H}_{p}$,

\begin{enumerate}
\item[{\bf (i)}] if $\xi \in {\cal V}_{p}$, then we have%
\begin{eqnarray}
\tau _{{\cal H}}^{{\cal H}}\left( p\right) -K_{{\cal H}}^{{\cal H}}({\Bbb P}%
) &\leq &\frac{c_{1}}{2}\left( s^{2}-s-2\right) +\frac{3}{2}c_{2}\left( \left\Vert P\right\Vert ^{2}-2\left(
g_{1}\left( U_{1},PU_{2}\right) \right) ^{2}\right),
\label{eq-hor-cf-GSSF(a)}
\end{eqnarray}

\item[{\bf (ii)}] if $\xi \in {\cal H}_{p}$, then we have%
\begin{eqnarray}
\tau _{{\cal H}}^{{\cal H}}\left( p\right) -K_{{\cal H}}^{{\cal H}}({\Bbb P}%
) &\leq &\frac{c_{1}}{2}\left( s^{2}-s-2\right)+\frac{3c_{2}}{2}\left( \left\Vert P\right\Vert ^{2}-2\left(
g_{1}\left( U_{1},PU_{2}\right) \right) ^{2}\right) \nonumber \\
&&\left. -c_{3}\left( \left( s-1\right) -{\bf \gamma }\left( {\Bbb P}\right)
\right) \right. ,  \label{eq-hor-cf-GSSF(b)}
\end{eqnarray}%
where ${\bf \gamma }\left( {\Bbb P}\right) =\left( \left( \eta \left(
U_{1}\right) \right) ^{2}+\left( \eta \left( U_{2}\right) \right)
^{2}\right) $. The equality holds in {\rm (\ref{eq-hor-cf-GSSF(a)}) }and 
{\rm (\ref{eq-hor-cf-GSSF(a)})} if and only if the tensor $A^{{\cal V}}$
satisfies {\rm (\ref{eq-hor-equality-1})} and {\rm (\ref{eq-hor-equality-2})}%
.
\end{enumerate}
\end{theorem}
\begin{proof}
Using (\ref{eq-P-(9.2.0)}), (\ref{eq-P-(9.2.1)}) and (%
\ref{eq-GSSF}), we obtain 
\begin{equation}
\tau _{{\cal H}}^{M_{1}}(p)=\left\{ 
\begin{array}{ll}
\frac{s\left( s-1\right) }{2}c_{1}+\frac{3}{2}c_{2}\left\Vert P\right\Vert
^{2} & {\rm if}~\xi \in {\cal V}_{p}; \\ 
\frac{s\left( s-1\right) }{2}c_{1}+\frac{3}{2}c_{2}\left\Vert P\right\Vert
^{2}-\left( s-1\right) c_{3} & {\rm if}~\xi \in {\rm {\cal H}}_{p},%
\end{array}%
\right.  \label{eq-scal-GSSF-Hor}
\end{equation}%
and 
\begin{equation}
K_{{\cal H}}^{M_{1}}\left( {\Bbb P}\right) =\left\{ 
\begin{array}{ll}
c_{1}+3c_{2}\left( g_{1}\left( U_{1},PU_{2}\right) \right) ^{2} & {\rm if}%
~\xi \in {\cal V}_{p}; \\ 
\begin{array}{l}
c_{1}+3c_{2}\left( g_{1}\left( U_{1},PU_{2}\right) \right) ^{2} \\ 
-c_{3}\left( \left( \eta \left( U_{1}\right) \right) ^{2}+\left( \eta \left(
U_{2}\right) \right) ^{2}\right)%
\end{array}
& {\rm if}~\xi \in {\rm {\cal H}}_{p}.%
\end{array}%
\right.  \label{eq-sec-GSSF-Hor}
\end{equation}%
Substituting the values from (\ref{eq-scal-GSSF-Hor}) and (\ref%
{eq-sec-GSSF-Hor}) into (\ref{eq-Ine-ScalH-scalM}), we directly obtain (\ref%
{eq-hor-cf-GSSF(a)}) and (\ref{eq-hor-cf-GSSF(b)}).
\end{proof}
\begin{example}
Consider the Riemannian submersion constructed in {\rm \cite[Example~32]%
{Singh_Meena_Meena_26_JMAA}}. It can be verified that this example satisfies
all the assumptions of {\rm Theorem~\ref{Th_GSSF_hor}}. Moreover, the
inequality established in this theorem is attained as an equality for this
case.
\end{example}
\section{B.-Y. Chen inequality for Riemannian submersion along mixed distributions \label{Section 4}}
We begin with the following:
\begin{theorem}
\label{Theorem GCFVH}Let $\pi :(M_{1},g_{1})\rightarrow (M_{2},g_{2})$ be a
Riemannian submersion between Riemannian manifolds with $\dim M_{1}=n$ and $%
\dim M_{2}=m$. If $\dim {\cal V}_{p}=r>2$, $\dim {\cal H}_{p}=s>2$, then for
any $2$-planes $\Pi \subset {\cal V}_{p}$ and ${\Bbb P}\subset {\cal H}_{p}$%
, 
\begin{eqnarray}
\tau ^{M_{1}}(p)-K_{{\cal V}}^{M_{1}}(\Pi )-K_{{\cal H}}^{M_{1}}({\Bbb P})
&\leq &\tau _{{\cal H}}^{{\cal H}}(p)-K_{{\cal H}}^{{\cal H}}({\Bbb P})+\tau
_{{\cal V}}^{{\cal V}}(p)-K_{{\cal V}}^{{\cal V}}(\Pi )  \nonumber \\
&&+\frac{r^{2}\left( r-2\right) }{2\left( r-1\right) }\left\Vert
H\right\Vert ^{2}-\breve{\delta}\left( N\right) +3\sum_{j=3}^{s}\sum_{\alpha
=1}^{r}\left( \left( {\cal A}^{{\cal V}}\right) _{1j}^{\alpha }\right) ^{2} 
\nonumber \\
&&+\frac{3}{2}\sum_{i,j=2}^{s}\sum_{\alpha =1}^{r}\left( \left( {\cal A}^{%
{\cal V}}\right) _{ij}^{\alpha }\right) ^{2}-\frac{1}{2}\left\Vert {\cal A}^{%
{\cal H}}\right\Vert +^{2}\frac{1}{2}\left\Vert T^{{\cal V}}\right\Vert ^{2},
\label{eq-GIVHCF}
\end{eqnarray}%
where $\tau ^{M_{1}}(p)=\tau _{{\cal V}}^{M_{1}}(p)+\tau _{{\cal H}%
}^{M_{1}}(p)+\sum_{i=1}^{s}\sum_{j=1}^{r}R^{M_{1}}\left(
U_{i},V_{j},V_{j},U_{i}\right) $. The equality holds in {\rm (\ref{eq-GIVHCF}%
)} if and only if the tensor ${\cal T}^{{\cal H}}$ takes the form given by 
{\rm (\ref{eq-TH-matrix-1})} and {\rm (\ref{eq-TH-matrix-2})}.
\end{theorem}
\begin{proof} We have the scalar curvature $\tau ^{M_{1}}(p)$ of $%
M_{1}$ \cite{Aytimur_Ozgur_21_JGP} 
\begin{eqnarray}
\tau ^{M_{1}}(p) &=&\sum_{1\leq i<j\leq r}R^{M_{1}}\left(
V_{i},V_{j},V_{j},V_{i}\right) +\sum_{1\leq i<j\leq s}R^{M_{1}}\left(
U_{i},U_{j},U_{j},U_{i}\right)   \nonumber \\
&&+\sum_{i=1}^{s}\sum_{j=1}^{r}R^{M_{1}}\left(
U_{i},V_{j},V_{j},U_{i}\right) .  \label{eq-TMVH-(1)}
\end{eqnarray}%
By using (\ref{eq-P-(9.2.1)}) in (\ref{eq-TMVH-(1)}), we have 
\begin{equation}
\tau ^{M_{1}}(p)=\tau _{{\cal V}}^{M_{1}}(p)+\tau _{{\cal H}%
}^{M_{1}}(p)+\sum_{i=1}^{s}\sum_{j=1}^{r}R^{M_{1}}\left(
U_{i},V_{j},V_{j},U_{i}\right) .  \label{eq-TMVH-(1.1)}
\end{equation}%
On the other hand, from (\ref{eq-P-(10)}), (\ref{eq-P-(11)}), (\ref%
{eq-P-(12)}), (\ref{eq-TMVH-(1.1)}) and using (\ref{eq-P-(9)}) and (\ref%
{eq-P-(9.1)}), we obtain 
\begin{eqnarray}
2\tau ^{M_{1}}(p) &=&2\tau _{{\cal H}}^{{\cal H}}(p)+2\tau _{{\cal V}}^{%
{\cal V}}(p)+r^{2}\left\Vert H\right\Vert ^{2}+3\sum_{i,j=1}^{s}g\left(
A_{U_{i}}U_{j},A_{U_{i}}U_{j}\right)   \nonumber \\
&&-\sum_{t=1}^{s}\sum_{i,j=1}^{r}\left( T_{ij}^{t}\right)
^{2}-2\sum_{j=1}^{r}\sum_{i=1}^{s}g\left( \left( \nabla _{U_{i}}T\right)
\left( V_{j},V_{j}\right) ,U_{i}\right)   \nonumber \\
&&+\sum_{j=1}^{r}\sum_{i=1}^{s}\left\{ g\left(
T_{V_{j}}U_{i},T_{V_{j}}U_{i}\right) -g\left(
A_{U_{i}}V_{j},A_{U_{i}}V_{j}\right) \right\} .  \label{eq-TMVH-(2)}
\end{eqnarray}%
By using (\ref{eq-P-(9.3.1)}), (\ref{eq-P-(14)}) and (\ref{eq-P-(14.1)}), we
obtain 
\begin{equation}
2\tau ^{M_{1}}(p)=2\tau _{{\cal H}}^{{\cal H}}(p)+2\tau _{{\cal V}}^{{\cal V}%
}(p)+r^{2}\left\Vert H\right\Vert ^{2}+3\left\Vert {\cal A}^{{\cal V}%
}\right\Vert ^{2}-\left\Vert T^{{\cal H}}\right\Vert ^{2}-2\breve{\delta}%
\left( N\right) +\left\Vert T^{{\cal V}}\right\Vert ^{2}-\left\Vert {\cal A}%
^{{\cal H}}\right\Vert ^{2}.  \label{eq-TMVH-(3)}
\end{equation}%
Define 
\begin{equation}
\varepsilon =2\tau ^{M_{1}}(p)-2\tau _{{\cal H}}^{{\cal H}}(p)-2\tau _{{\cal %
V}}^{{\cal V}}(p)-\frac{r^{2}\left( r-2\right) }{\left( r-1\right) }%
\left\Vert H\right\Vert ^{2}-3\left\Vert {\cal A}^{{\cal V}}\right\Vert
^{2}+2\breve{\delta}\left( N\right) -\left\Vert T^{{\cal V}}\right\Vert
^{2}+\left\Vert {\cal A}^{{\cal H}}\right\Vert ^{2}.  \label{eq-TMVH-(4)}
\end{equation}%
From (\ref{eq-TMVH-(3)}) and (\ref{eq-TMVH-(4)}), it follows that 
\begin{equation}
r^{2}\Vert H\Vert ^{2}=(r-1)\left( \varepsilon +\left\Vert T^{{\cal H}%
}\right\Vert ^{2}\right) .  \label{eq-TMVH-(5)}
\end{equation}%
By carrying out computations similar to those presented in Theorem~\ref%
{Theorem 1} (from equation~(\ref{eq-(5)}) onward), we obtain 
\begin{equation}
K_{{\cal V}}^{{\cal V}}(\Pi )\leq K_{{\cal V}}^{M_{1}}(\Pi )-\tfrac{%
\varepsilon }{2}.  \label{eq-TMVH-(8)}
\end{equation}%
From (\ref{eq-P-(11)}), (\ref{eq-P-(9.2.0)}) and (\ref{eq-P-(14.1)}), we
have 
\begin{equation}
\frac{3}{2}\left\Vert {\cal A}^{{\cal V}}\right\Vert ^{2}=K_{{\cal H}%
}^{M_{1}}({\Bbb P})-K_{{\cal H}}^{{\cal H}}({\Bbb P})+3\sum_{j=3}^{s}\sum_{%
\alpha =1}^{r}\left( \left( {\cal A}^{{\cal V}}\right) _{1j}^{\alpha
}\right) ^{2}+\frac{3}{2}\sum_{i,j=2}^{s}\sum_{\alpha =1}^{r}\left( \left( 
{\cal A}^{{\cal V}}\right) _{ij}^{\alpha }\right) ^{2}.  \label{eq-TMVH-(9)}
\end{equation}%
In view of (\ref{eq-TMVH-(4)}), (\ref{eq-TMVH-(8)}) and (\ref{eq-TMVH-(9)}),
we get (\ref{eq-GIVHCF}). If the equality in (\ref{eq-GIVHCF}) holds, then
the inequalities given by (\ref{eq-(6)}) and (\ref{eq-TMVH-(8)}) become
equalities. In this case, we observe that the tensor ${\cal T}^{{\cal H}}$
takes the form given by {\rm (\ref{eq-TH-matrix-1})} and {\rm (\ref%
{eq-TH-matrix-2})}.
\end{proof}
\begin{remark}
In {\rm Examples \ref{Exa-GIRMEDNH}} and {\rm \ref{Exa-GIGSEH}}, we observe
that the maps $\pi $ satisfy the assumptions of {\rm Theorem \ref{Theorem
GCFVH}}. Moreover, the inequality derived in {\rm Theorem \ref{Theorem GCFVH}%
} attains equality in {\rm Example \ref{Exa-GIGSEH}}, while it does not
achieve equality in {\rm Example \ref{Exa-GIRMEDNH}}.
\end{remark}
Now, we apply {\rm Theorem \ref{Theorem GCFVH}} to various cases of
Riemannian submersions, as given below.
\begin{theorem}
\label{Theorem RSFCFHV}Let $\pi :(M_{1},g_{1})\rightarrow (M_{2},g_{2})$ be
a Riemannian submersions from real space form of constant sectional
curvature $c$ onto a Riemannian manifold with $\dim M_{1}=n$ and $\dim
M_{2}=m$. If $\dim {\cal V}_{p}=r>2$, $\dim {\cal H}_{p}=s>2$, then for any $%
2$-planes $\Pi \subset {\cal V}_{p}$ and ${\Bbb P}\subset {\cal H}_{p}$, 
\begin{eqnarray}
\frac{c}{2}\left( r^{2}+s^{2}+2sr-s-r-4\right)  &\leq &\tau _{{\cal H}}^{%
{\cal H}}-K_{{\cal H}}^{{\cal H}}({\Bbb P})+\tau _{{\cal V}}^{{\cal V}}-K_{%
{\cal V}}^{{\cal V}}(\Pi )  \nonumber \\
&&+\frac{r^{2}\left( r-2\right) }{2\left( r-1\right) }\left\Vert
H\right\Vert ^{2}-\breve{\delta}\left( N\right) +3\sum_{j=3}^{s}\sum_{\alpha
=1}^{r}\left( \left( {\cal A}^{{\cal V}}\right) _{1j}^{\alpha }\right) ^{2} 
\nonumber \\
&&+\frac{3}{2}\sum_{i,j=2}^{s}\sum_{\alpha =1}^{r}\left( \left( {\cal A}^{%
{\cal V}}\right) _{ij}^{\alpha }\right) ^{2}-\frac{1}{2}\left\Vert {\cal A}^{%
{\cal H}}\right\Vert ^{2}+\frac{1}{2}\left\Vert T^{{\cal V}}\right\Vert ^{2}.
\label{eqVHRSF-(1)}
\end{eqnarray}%
The equality holds in {\rm (\ref{eqVHRSF-(1)})} if and only if the tensor $%
{\cal T}^{{\cal H}}$ takes the form given by {\rm (\ref{eq-TH-matrix-1})}
and {\rm (\ref{eq-TH-matrix-2})}.
\end{theorem}
\begin{proof}
By (\ref{eq-RSF}), we get 
\begin{equation}
\sum_{i=1}^{s}\sum_{j=1}^{r}R^{M_{1}}\left( U_{i},V_{j},V_{j},U_{i}\right)
=csr.  \label{eq-VHRSF-(1.1)}
\end{equation}%
In view of (\ref{eq-RSF-(1.1)}), (\ref{eq-scal-real}), (\ref{eq-GIVHCF}) and
(\ref{eq-VHRSF-(1.1)}), we get (\ref{eqVHRSF-(1)}).
\end{proof}
\begin{remark}
In {\rm Example \ref{Exa-RSFEH}}, we see that the map $\pi $ satisfies the
assumptions of {\rm Theorem \ref{Theorem RSFCFHV}}. Moreover, the inequality
given in {\rm Theorem \ref{Theorem RSFCFHV}} attains equality in this
example.
\end{remark}
\begin{theorem}
\label{Theorem CSFCFHV}Let $\pi :\left( {M_{1}}(c),g_{1}\right) \rightarrow
\left( M_{2},g_{2}\right) $ be a Riemannian submersion from a complex space
form of constant holomorphic sectional curvature $c$ onto a Riemannian
manifold with $\dim M_{1}=n=2k$, $\dim M_{2}=m.$ If $\dim {\cal V}_{p}=r>2$, 
$\dim {\cal H}_{p}=s>2$, then for any $2$-planes $\Pi \subset {\cal V}_{p}$
and ${\Bbb P}\subset {\cal H}_{p}$, 
\begin{eqnarray}
&&\left. \frac{c}{8}\left( r^{2}+s^{2}+2sr-s-r-4\right) +\frac{3c}{8}\left(
\left\Vert Q\right\Vert ^{2}+\left\Vert P\right\Vert ^{2}+2\left\Vert P^{%
{\cal V}}\right\Vert ^{2}\right. \right.   \nonumber \\
&&\left. \left. -2\left( g_{1}\left( V_{1},QV_{2}\right) \right)
^{2}-2\left( g_{1}\left( U_{1},PU_{2}\right) \right) ^{2}\right) \leq \tau _{%
{\cal H}}^{{\cal H}}-K_{{\cal H}}^{{\cal H}}({\Bbb P})+\tau _{{\cal V}}^{%
{\cal V}}-K_{{\cal V}}^{{\cal V}}(\Pi )\right.   \nonumber \\
&&\left. +\frac{r^{2}\left( r-2\right) }{2\left( r-1\right) }\left\Vert
H\right\Vert ^{2}-\breve{\delta}\left( N\right) +3\sum_{j=3}^{s}\sum_{\alpha
=1}^{r}\left( \left( {\cal A}^{{\cal V}}\right) _{1j}^{\alpha }\right) ^{2}+%
\frac{3}{2}\sum_{i,j=2}^{s}\sum_{\alpha =1}^{r}\left( \left( {\cal A}^{{\cal %
V}}\right) _{ij}^{\alpha }\right) ^{2}\right.   \nonumber \\
&&\left. -\frac{1}{2}\left\Vert {\cal A}^{{\cal H}}\right\Vert ^{2}+\frac{1}{%
2}\left\Vert T^{{\cal V}}\right\Vert ^{2},\right.   \label{eqVHCSF-(1)}
\end{eqnarray}%
where $\left\Vert P^{{\cal V}}\right\Vert ^{2}=\sum_{j=1}^{r}\left\Vert
PV_{j}\right\Vert ^{2}=\sum_{i=1}^{s}\sum_{j=1}^{r}\left(
g(U_{i},JV_{j})\right) ^{2}$. The equality holds in {\rm (\ref{eqVHCSF-(1)})}
if and only if the tensor ${\cal T}^{{\cal H}}$ takes the form given by {\rm %
(\ref{eq-TH-matrix-1})} and {\rm (\ref{eq-TH-matrix-2})}.
\end{theorem}
\begin{proof}
By (\ref{eq-GCSF}), we obtain%
\begin{equation}
\sum_{i=1}^{s}\sum_{j=1}^{r}R^{M_{1}}\left( U_{i},V_{j},V_{j},U_{i}\right) =%
\frac{c}{4}sr+\frac{3c}{4}\left\Vert P^{{\cal V}}\right\Vert ^{2}.
\label{eqVHCSF-(2)}
\end{equation}%
In view of (\ref{eq-GCSF-(1.1)}), (\ref{eq-scal-complex}), (\ref{eq-GIVHCF})
and (\ref{eqVHCSF-(2)}), we get (\ref{eqVHCSF-(1)}).
\end{proof}
\begin{remark}
In Example {\rm \ref{Exa-CSFEH}}, we see that the map $\pi $ satisfies the
assumptions of {\rm Theorem \ref{Theorem CSFCFHV}}. Moreover, the inequality
given in {\rm Theorem \ref{Theorem CSFCFHV}} attains equality in this
example.
\end{remark}
\begin{theorem}
\label{Theorem GSSFCFHV}Let $\pi :\left( {M_{1}}(c_{1},c_{2},c_{3}),g_{1}%
\right) \rightarrow \left( M_{2},g_{2}\right) $ be a Riemannian submersion
from a generalized Sasakian space form onto a Riemannian manifold with $\dim
M_{1}=n=2k+1$ and $\dim M_{2}=m$. If $\dim {\cal V}_{p}=r>2$, $\dim {\cal H}%
_{p}=s>2$, then for any $2$-planes $\Pi \subset {\cal V}_{p}$ and ${\Bbb P}%
\subset {\cal H}_{p}$,

\begin{enumerate}
\item[{\bf (i)}] if $\xi \in {\cal V}_{p}$, then we have%
\begin{eqnarray}
&&\left. \frac{c_{1}}{2}\left( r^{2}+s^{2}+2sr-s-r-4\right) +\frac{3c_{2}}{2}%
\left( \left\Vert Q\right\Vert ^{2}+\left\Vert P\right\Vert ^{2}+2\left\Vert
P^{{\cal V}}\right\Vert ^{2}-2\left( g_{1}\left( V_{1},QV_{2}\right) \right)
^{2}\right. \right.   \nonumber \\
&&\left. \left. -2\left( g_{1}\left( U_{1},PU_{2}\right) \right) ^{2}\right)
-c_{3}\left( r+s-1-\Theta (\Pi )\right) \leq \tau _{{\cal H}}^{{\cal H}}-K_{%
{\cal H}}^{{\cal H}}({\Bbb P})+\tau _{{\cal V}}^{{\cal V}}-K_{{\cal V}}^{%
{\cal V}}(\Pi )\right.   \nonumber \\
&&\left. +\frac{r^{2}\left( r-2\right) }{2\left( r-1\right) }\left\Vert
H\right\Vert ^{2}-\breve{\delta}\left( N\right) +3\sum_{j=3}^{s}\sum_{\alpha
=1}^{r}\left( \left( {\cal A}^{{\cal V}}\right) _{1j}^{\alpha }\right) ^{2}+%
\frac{3}{2}\sum_{i,j=2}^{s}\sum_{\alpha =1}^{r}\left( \left( {\cal A}^{{\cal %
V}}\right) _{ij}^{\alpha }\right) ^{2}\right.   \nonumber \\
&&\left. -\frac{1}{2}\left\Vert {\cal A}^{{\cal H}}\right\Vert ^{2}+\frac{1}{%
2}\left\Vert T^{{\cal V}}\right\Vert ^{2},\right.   \label{eq-VHGSSF-(1)}
\end{eqnarray}

\item[{\bf (ii)}] if $\xi \in {\cal H}_{p}$, then we have%
\begin{eqnarray}
&&\left. \frac{c_{1}}{2}\left( r^{2}+s^{2}+2sr-s-r-4\right) +\frac{3c_{2}}{2}%
\left( \left\Vert Q\right\Vert ^{2}+\left\Vert P\right\Vert ^{2}+2\left\Vert
P^{{\cal V}}\right\Vert ^{2}-2\left( g_{1}\left( V_{1},QV_{2}\right) \right)
^{2}\right. \right.   \nonumber \\
&&\left. \left. -2\left( g_{1}\left( U_{1},PU_{2}\right) \right) ^{2}\right)
-c_{3}\left( s+r-1-{\bf \gamma }\left( {\Bbb P}\right) \right) \leq \tau _{%
{\cal H}}^{{\cal H}}-K_{{\cal H}}^{{\cal H}}({\Bbb P})+\tau _{{\cal V}}^{%
{\cal V}}-K_{{\cal V}}^{{\cal V}}(\Pi )\right.   \nonumber \\
&&\left. +\frac{r^{2}\left( r-2\right) }{2\left( r-1\right) }\left\Vert
H\right\Vert ^{2}-\breve{\delta}\left( N\right) +3\sum_{j=3}^{s}\sum_{\alpha
=1}^{r}\left( \left( {\cal A}^{{\cal V}}\right) _{1j}^{\alpha }\right) ^{2}+%
\frac{3}{2}\sum_{i,j=2}^{s}\sum_{\alpha =1}^{r}\left( \left( {\cal A}^{{\cal %
V}}\right) _{ij}^{\alpha }\right) ^{2}\right.   \nonumber \\
&&\left. -\frac{1}{2}\left\Vert {\cal A}^{{\cal H}}\right\Vert ^{2}+\frac{1}{%
2}\left\Vert T^{{\cal V}}\right\Vert ^{2}.\right.   \label{eq-VHGSSF-(2)}
\end{eqnarray}
\end{enumerate}
The equality holds in {\rm (\ref{eq-VHGSSF-(1)})} and {\rm (\ref%
{eq-VHGSSF-(2)})} if and only if the tensor ${\cal T}^{{\cal H}}$ takes the
form given by {\rm (\ref{eq-TH-matrix-1})} and {\rm (\ref{eq-TH-matrix-2})}.
\end{theorem}
\begin{proof}
By (\ref{eq-GSSF}), we obtain 
\begin{equation}
\sum_{i=1}^{s}\sum_{j=1}^{r}R^{M_{1}}\left( U_{i},V_{j},V_{j},U_{i}\right)
=\left\{ 
\begin{array}{ll}
c_{1}sr+3c_{2}\left\Vert P^{{\cal V}}\right\Vert ^{2}-c_{3}s & {\rm if}~\xi
\in {\cal V}_{p}; \\ 
c_{1}sr+3c_{2}\left\Vert P^{{\cal V}}\right\Vert ^{2}-c_{3}r & {\rm if}~\xi
\in {\rm {\cal H}}_{p}.%
\end{array}%
\right.  \label{eq-VHGSSF-(6)}
\end{equation}%
In view of (\ref{eq-GSSFMI-(1.1)}), (\ref{eq-GSSFMI-(1.2)}), (\ref%
{eq-scal-GSSF-Hor}), (\ref{eq-sec-GSSF-Hor}), (\ref{eq-GIVHCF}) and (\ref%
{eq-VHGSSF-(6)}), we get (\ref{eq-VHGSSF-(1)}) and (\ref{eq-VHGSSF-(2)}). 
\end{proof}
\begin{remark}
In {\rm Example \ref{Exa-GSSF}}, we see that the map $\pi $ satisfies the
assumptions of {\rm Theorem~\ref{Theorem GSSFCFHV}}. Moreover, the
inequality given in {\rm Theorem \ref{Theorem GSSFCFHV}} attains equality in
this example.
\end{remark}

\end{document}